\documentclass[a4paper,11pt,fleqn]{article}
\usepackage{pstricks}
\usepackage{amsmath}
\usepackage{amssymb}
\usepackage{theorem} 
\usepackage{euscript}
\usepackage{exscale,relsize}
\usepackage{graphicx}
\usepackage{srcltx}
\usepackage{xcolor}
\usepackage{url}
\usepackage{charter}
\newcommand{\email}[1]{\href{mailto:#1}{\nolinkurl{#1}}}
\renewcommand{\leq}{\ensuremath{\leqslant}}
\renewcommand{\geq}{\ensuremath{\geqslant}}

\newcommand{\minimize}[2]{\ensuremath{\underset{\substack{{#1}}}%
{\text{minimize}}\;\;#2 }}
\newcommand{\moyo}[2]{\ensuremath{\sideset{^{#2}}{}{\operatorname{}}\!\!#1}}
 
\newcommand{\Scal}[2]{\bigg\langle{#1}\;\bigg|\:{#2}\bigg\rangle} 
\newcommand{\scal}[2]{{\left\langle{{#1}\mid{#2}}\right\rangle}}
\newcommand{\menge}[2]{\big\{{#1}~\big |~{#2}\big\}} 
\newcommand{\Menge}[2]{\left\{{#1}~\bigg|~{#2}\right\}} 
\newcommand{\HHH}{{\ensuremath{\boldsymbol{\mathcal H}}}}
\newcommand{\KKK}{\ensuremath{\boldsymbol{\mathcal K}}}
\newcommand{\GGG}{\ensuremath{{\boldsymbol{\mathcal G}}}}
\newcommand{\HH}{\ensuremath{{\mathcal H}}}
\newcommand{\GG}{\ensuremath{{\mathcal G}}}
\newcommand{\Sum}{\ensuremath{\displaystyle\sum}}
\newcommand{\emp}{\ensuremath{{\varnothing}}}

\newcommand{\Id}{\ensuremath{\text{Id}}\,}
\newcommand{\cart}{\ensuremath{\raisebox{-0.5mm}{\mbox{\LARGE{$\times$}}}}}

\newcommand{\RR}{\ensuremath{\mathbb{R}}}
\newcommand{\RP}{\ensuremath{\left[0,+\infty\right[}}

\newcommand{\BL}{\ensuremath{\EuScript B}\,}
\newcommand{\RPP}{\ensuremath{\left]0,+\infty\right[}}

\newcommand{\RPX}{\ensuremath{\left[0,+\infty\right]}}
\newcommand{\RX}{\ensuremath{\left]-\infty,+\infty\right]}}
\newcommand{\RXX}{\ensuremath{\left[-\infty,+\infty\right]}}
\newcommand{\NN}{\ensuremath{\mathbb N}}

\newcommand{\weakly}{\ensuremath{\:\rightharpoonup\:}}
\newcommand{\exi}{\ensuremath{\exists\,}}
\newcommand{\ran}{\ensuremath{\text{\rm ran}}}
\newcommand{\zer}{\ensuremath{\text{\rm zer}\,}}
\newcommand{\pinf}{\ensuremath{{+\infty}}}

\newcommand{\dom}{\ensuremath{\text{\rm dom}\,}}
\newcommand{\prox}{\ensuremath{\text{\rm prox}}}

\newcommand{\spa}{\ensuremath{\text{\rm span}\,}}

\newcommand{\gra}{\ensuremath{\text{\rm gra}}}

\newcommand{\sri}{\ensuremath{\text{\rm sri}\,}}
\newcommand{\reli}{\ensuremath{\text{\rm ri}\,}}
\newcommand{\infconv}{\ensuremath{\mbox{\small$\,\square\,$}}}

\newcommand{\zeroun}{\ensuremath{\left]0,1\right[}}

\newtheorem{theorem}{Theorem}[section]
\newtheorem{lemma}[theorem]{Lemma}

\newtheorem{proposition}[theorem]{Proposition}

\theoremstyle{plain}{\theorembodyfont{\rmfamily}%
}
\theoremstyle{plain}{\theorembodyfont{\rmfamily}%
\newtheorem{example}[theorem]{Example}}
\theoremstyle{plain}{\theorembodyfont{\rmfamily}%
\newtheorem{remark}[theorem]{Remark}}
\theoremstyle{plain}{\theorembodyfont{\rmfamily}%
}
\theoremstyle{plain}{\theorembodyfont{\rmfamily}%
}
\theoremstyle{plain}{\theorembodyfont{\rmfamily}%
}
\theoremstyle{plain}{\theorembodyfont{\rmfamily}%
\newtheorem{problem}[theorem]{Problem}}

\numberwithin{equation}{section}
\begin{document}

\title{\sffamily\huge Systems of Structured Monotone Inclusions:
Duality, Algorithms, and Applications\footnote{Contact author: 
P. L. Combettes, {\ttfamily plc@math.jussieu.fr},
phone: +33 1 4427 6319, fax: +33 1 4427 7200.}}

\author{Patrick L. Combettes\\[5mm]
\small UPMC Universit\'e Paris 06\\
\small Laboratoire Jacques-Louis Lions -- UMR CNRS 7598\\
\small 75005 Paris, France\\
\small {plc@math.jussieu.fr}\\[4mm]
}

\date{~}

\maketitle

\vskip 8mm

\begin{abstract} \noindent
A general primal-dual splitting algorithm for solving systems of 
structured coupled monotone inclusions in Hilbert spaces is 
introduced and its asymptotic behavior is analyzed. 
Each inclusion in the primal system 
features compositions with linear operators, parallel sums, and 
Lipschitzian operators. All the operators involved in this 
structured model are used separately in the proposed algorithm, 
most steps of which can be executed in parallel. This provides 
a flexible solution method applicable to a variety of problems 
beyond the reach of the state-of-the-art. Several applications are 
discussed to illustrate this point.
\end{abstract} 



\maketitle

\section{Introduction}

Traditional monotone operator splitting techniques 
\cite{Livre1,Opti04,Ecks92,Facc03,Glow89,Lion79,Merc79,%
Spin83,Tsen91,Tsen00} have their roots in matrix decomposition 
methods in numerical analysis \cite{Doug56,Varg00} and in nonlinear 
methods for solving optimization and variational inequality 
problems \cite{Baku74,Brez67,Korp76,Levi66,Sibo70}.
These methods are designed to solve inclusions of the 
type $0\in B_1x+B_2x$, where $B_1$ and $B_2$ are maximally monotone
operators acting on a Hilbert space $\HH$. Extensions to 
sums of the type $0\in\sum_{k=1}^KB_kx$ are typically handled 
via reformulations in product spaces \cite{Livre1,Spin83}. 
In recent years, new splitting algorithms have emerged for 
problems involving more complex models featuring compositions 
with linear operators \cite{Siop11} and parallel sums 
\cite{Svva12,Bang12} (see \eqref{e:parasum}). 
These algorithms rely on
reformulations of the inclusions as two-operator problems
in a primal-dual space, in which the splitting is performed via 
an existing method. This construct makes it possible to activate
separately each of the operators present in
the model, and it leads to flexible algorithms implementable on 
parallel architectures. In the present paper, we pursue this
strategy towards more sophisticated models involving systems of 
structured coupled inclusions in duality. The primal-dual
problem under consideration is the following.

\begin{problem}
\label{prob:1}
Let $m$ and $K$ be strictly positive integers, let
$(\HH_i)_{1\leq i\leq m}$ and $(\GG_k)_{1\leq k\leq K}$ be real 
Hilbert spaces, let $(\mu_i)_{1\leq i\leq m}\in\RP^m$, and let 
$(\nu_k)_{1\leq i\leq K}\in\RP^K$. For every $i\in\{1,\ldots,m\}$ 
and $k\in\{1,\ldots,K\}$, let $C_i\colon\HH_i\to\HH_i$ be 
monotone and $\mu_i$-Lipschitzian, let 
$A_i\colon\HH_i\to 2^{\HH_i}$ and $B_k\colon\GG_k\to 2^{\GG_k}$ 
be maximally monotone, let $D_k\colon\GG_k\to 2^{\GG_k}$ be 
maximally monotone and such that $D_k^{-1}\colon\GG_k\to\GG_k$ is 
$\nu_k$-Lipschitzian, let $z_i\in\HH_i$, let $r_k\in\GG_k$, 
and let $L_{ki}\in\BL(\HH_i,\GG_k)$. It is assumed that 
\begin{equation}
\label{e:5h8Njiq-11a}
\beta=\text{max}\Big\{\underset{1\leq i\leq m}{\text{max}}\mu_i,
\underset{1\leq k\leq K}{\text{max}}\nu_k\Big\}+
\sqrt{\lambda}>0,
\quad\text{where}\quad
\lambda\in\left[\underset{\sum_{i=1}^m\|x_i\|^2\leq 1}{\text{sup}}
\sum_{k=1}^K\bigg\|\sum_{i=1}^mL_{ki}x_i\bigg\|^2,\pinf\right[,
\end{equation}
and that the system of coupled inclusions
\begin{multline}
\label{e:IJ834hj8fr-24p}
\text{find}\;\;\overline{x_1}\in\HH_1,\ldots,\overline{x_m}\in\HH_m
\;\;\text{such that}\\
\begin{cases}
z_1&\!\!\!\in A_1\overline{x_1}+\Sum_{k=1}^KL_{k1}^*
\bigg((B_k\infconv
D_k)\bigg(\Sum_{i=1}^mL_{ki}\overline{x_i}-r_k\bigg)\bigg)
+C_1\overline{x_1}\\
&\!\!\!\;\vdots\\[-3mm]
z_m&\!\!\!\in A_m\overline{x_m}+\Sum_{k=1}^KL_{km}^*
\bigg((B_k\infconv D_k)\bigg(\Sum_{i=1}^m
L_{ki}\overline{x_i}-r_k\bigg)\bigg)+C_m\overline{x_m}
\end{cases}
\end{multline}
possesses at least one solution. Solve \eqref{e:IJ834hj8fr-24p} 
together with the dual problem
\begin{multline}
\label{e:IJ834hj8fr-24d}
\text{find}\;\;\overline{v_1}\in\GG_1,\ldots,\overline{v_K}
\in\GG_K\;\;\text{such that}\\
\begin{cases}
-r_1&\!\!\!
\in-\Sum_{i=1}^mL_{1i}\big(A_i+C_i\big)^{-1}
\bigg(z_i-\Sum_{k=1}^KL_{ki}^*\overline{v_k}\bigg)
+B_1^{-1}\overline{v_1}+D_1^{-1}\overline{v_1}\\
&\!\!\!\;\vdots\\[-3mm]
-r_K&\!\!\!\in-\Sum_{i=1}^mL_{Ki}\big(A_i+C_i\big)^{-1}
\bigg(z_i-\Sum_{k=1}^KL_{ki}^*\overline{v_k}\bigg)
+B_K^{-1}\overline{v_K}+D_K^{-1}\overline{v_K}.
\end{cases}
\end{multline}
\end{problem}

The primal system \eqref{e:IJ834hj8fr-24p} captures a broad class of 
problems in nonlinear analysis in which $m$ variables $x_1$, 
\dots, $x_m$ interact. The $i$th inclusion in 
\eqref{e:IJ834hj8fr-24p} features two operators $A_i$ and $C_i$
which model some abstract utility of the variable $x_i$, while 
the operator $(B_k)_{1\leq k\leq K}$, $(D_k)_{1\leq k\leq K}$, and 
$(L_{ki})_{\substack{1\leq i\leq m\\ 1\leq k\leq K}}$ 
model the interaction between $x_i$ and the remaining variables. 
One of the simplest realizations of \eqref{e:IJ834hj8fr-24p} is 
the problem considered in \cite{Reic05}, namely
\begin{equation}
\label{e:prob63}
\text{find}\;\;\overline{x_1}\in\HH,\, \overline{x_2}
\in\HH\;\;\text{such that}\quad
\begin{cases}
0\in A_1\overline{x_1}+\overline{x_1}-\overline{x_2}\\
0\in A_2\overline{x_2}-\overline{x_1}+\overline{x_2},
\end{cases}
\end{equation}
where $(\HH,\|\cdot\|)$ is a real Hilbert space, and where 
$A_1$ and $A_2$ are maximally monotone operators acting on 
$\HH$. In particular, if $A_1=\partial f_1$ and 
$A_2=\partial f_2$ are the subdifferentials of proper lower 
semicontinuous convex functions $f_1$ and $f_2$ from $\HH$ 
to $\RX$, \eqref{e:prob63} becomes
\begin{equation}
\label{e:prob62}
\minimize{x_1\in\HH,\,x_2\in\HH}{f_1(x_1)+f_2(x_2)+
\frac{1}{2}\|x_1-x_2\|^2}.
\end{equation}
This formulation arises in areas such as optimization 
\cite{Acke80}, the cognitive sciences \cite{Atto07}, image 
recovery \cite{Smms05}, signal synthesis \cite{Gold85}, best 
approximation \cite{Baus04}, and mechanics \cite{Merc80}.
In \cite{Sico10}, we considered the extension of \eqref{e:prob62} 
which amounts to setting in Problem~\ref{prob:1}, for every 
$i\in\{1,\ldots,m\}$ and $k\in\{1,\ldots,K\}$, 
$A_i=\partial f_i$, $C_i=0$, and $B_k=\nabla g_k$,
where $f_i\colon\HH\to\RX$ is a proper lower semicontinuous
convex function and $g_k\colon\GG_k\to\RR$ is convex and
differentiable with a Lipschitzian gradient. This leads to the 
minimization problem
\begin{equation}
\label{e:genna07-4}
\minimize{x_1\in\HH_1,\ldots,\,x_m\in\HH_m}{\sum_{i=1}^mf_i(x_i)
+\sum_{k=1}^K g_k\bigg(\sum_{i=1}^mL_{ki}x_i\bigg)},
\end{equation}
which has numerous applications in signal processing, machine
learning, image recovery, partial differential equations, and game 
theory; see \cite{Atto08,Bach12,Nmtm09,Jmiv11,Fran12,Gill07,Star05}
and the references therein.
In the case when $m=1$ in Problem~\ref{prob:1}, and under certain
restrictions on the operators involved, primal-dual algorithms 
have been proposed recently in \cite{Siop11,Svva12,Bang12}.
On the other hand, a primal algorithm was proposed in \cite{Sico10}
for solving systems of inclusions of type \eqref{e:IJ834hj8fr-24p} 
in which the operators $(C_i)_{1\leq i\leq m}$ and 
$(D_k^{-1})_{1\leq k\leq K}$ are zero, and the 
coupling operators $(B_k)_{1\leq k\leq K}$ are restricted to be 
single-valued and to satisfy jointly a cocoercivity property. 

The goal of the present paper is to develop a flexible algorithm
to solve Problem~\ref{prob:1} without the restrictions imposed
by current methods. In particular, no additional hypotheses will
be placed neither on the coupling operators $(B_k)_{1\leq k\leq K}$
and $(D_k)_{1\leq k\leq K}$, nor on the number $m$ of variables. 
In the proposed parallel splitting algorithm, the structure of 
the problem is fully exploited to the 
extent that the operators are all used individually, either 
explicitly if they are single-valued, or by means of their 
resolvent if they are set-valued. The main algorithm 
is introduced and analyzed in Section~\ref{sec:3}. 
The remaining sections are devoted to applications to problems
which are not explicitly solvable via existing techniques. 
Thus, in Section~\ref{sec:4}, we discuss applications to univariate
inclusion problems featuring general parallel sums, 
in the sense that the operators $(D_k)_{1\leq k\leq K}$ need
not have Lipschitzian inverses. In Section~\ref{sec:5}, we
apply this framework to the regularization of inconsistent common
zero problems. Finally, Sections~\ref{sec:6} and \ref{sec:7} address,
respectively, applications to multivariate and univariate
structured convex minimization problems.

\noindent
{\bfseries Notation.}
We denote the scalar product of a Hilbert space by 
$\scal{\cdot}{\cdot}$ and the associated norm by $\|\cdot\|$.
The symbols $\weakly$ and $\to$ denote, respectively, weak and 
strong convergence, and $\Id$ denotes the identity operator. 
Let $\HH$ and $\GG$ be real Hilbert spaces
and let $2^{\HH}$ be the power set of $\HH$.
The space of bounded linear operators from $\HH$ to $\GG$ is 
denoted by $\BL(\HH,\GG)$. Let $A\colon\HH\to 2^{\HH}$.
We denote by $\ran A=\menge{u\in\HH}{(\exi x\in\HH)\;u\in Ax}$ 
the range $A$, by 
$\dom A=\menge{x\in\HH}{Ax\neq\emp}$ the domain of $A$, by 
$\zer A=\menge{x\in\HH}{0\in Ax}$ the set of zeros of $A$, 
by $\gra A=\menge{(x,u)\in\HH\times\HH}{u\in Ax}$ the graph of 
$A$, and by $A^{-1}$ the inverse of $A$, i.e., the operator 
with graph
$\menge{(u,x)\in\HH\times\HH}{u\in Ax}$. The resolvent of $A$ is
$J_A=(\Id+A)^{-1}$. Moreover, $A$ is declared monotone if
\begin{equation}
(\forall (x,u)\in\gra A)(\forall (y,v)\in\gra A)\quad
\scal{x-y}{u-v}\geq 0,
\end{equation}
and maximally monotone if there exists no monotone operator 
$B\colon\HH\to 2^{\HH}$ such that $\gra A\subset\gra B\neq\gra A$.
In this case, $J_A$ is a nonexpansive operator defined everywhere
on $\HH$. Furthermore, $A$ is uniformly monotone at 
$x\in\dom A$ if there exists an increasing function 
$\phi\colon\RP\to\RPX$ that vanishes only at $0$ such that
\begin{equation}
\label{e:genova}
(\forall u\in Ax)(\forall (y,v)\in\gra A)\quad
\scal{x-y}{u-v}\geq\phi(\|x-y\|),
\end{equation}
and $A$ is couniformly monotone at $u\in\ran A$ if $A^{-1}$ 
is uniformly monotone at $u$.
The parallel sum of $A$ and $B\colon\HH\to 2^{\HH}$ is
\begin{equation}
\label{e:parasum}
A\infconv B=(A^{-1}+B^{-1})^{-1}.
\end{equation}
The infimal convolution of two functions $g$ and 
$\ell$ from $\HH$ to $\RX$ is
\begin{equation}
\label{e:infconv1}
g\infconv\ell\colon\HH\to\RXX\colon x\mapsto
\inf_{y\in\HH}\big(g(y)+\ell(x-y)\big).
\end{equation}
We denote by $\Gamma_0(\HH)$ the class of lower semicontinuous 
convex functions $f\colon\HH\to\RX$ such that
$\dom f=\menge{x\in\HH}{f(x)<\pinf}\neq\emp$. Let
$f\in\Gamma_0(\HH)$. The conjugate of $f$ is 
$\Gamma_0(\HH)\ni f^*\colon u\mapsto
\sup_{x\in\HH}(\scal{x}{u}-f(x))$, and $f$ is 
uniformly convex at $x\in\dom f$ if there exists an increasing 
function $\phi\colon\RP\to\RPX$ that vanishes only at $0$ such that
\begin{equation}
\label{e:Unifconvex}
(\forall y\in\dom f)(\forall\alpha\in\zeroun)\quad
f(\alpha x+(1-\alpha)y)+\alpha(1-\alpha)\phi(\|x-y\|)\leq
\alpha f(x)+(1-\alpha)f(y).
\end{equation}
For every $x\in\HH$, $f+\|x-\cdot\|^2/2$ possesses a 
unique minimizer, which is denoted by $\prox_fx$. We have 
\begin{equation}
\label{e:prox2}
\prox_f=J_{\partial f},\quad\text{where}\quad
\partial f\colon\HH\to 2^{\HH}\colon x\mapsto
\menge{u\in\HH}{(\forall y\in\HH)\;\:\scal{y-x}{u}+f(x)\leq f(y)} 
\end{equation}
is the subdifferential of $f$. Let $C$ be a convex subset of $\HH$.
The indicator function of $C$ is denoted by $\iota_C$ and the
distance function to $C$ by $d_C$.
The relative interior [respectively, the strong relative 
interior] of $C$, i.e., the set of points $x\in C$ 
such that the cone generated by $-x+C$ is a vector subspace 
[respectively, closed vector subspace] of $\HH$, by $\reli C$
[respectively, $\sri C$]. See \cite{Livre1,Zali02} for
background on convex analysis and monotone operators. 

\section{General algorithm}
\label{sec:3}

We start with three preliminary results. The first one 
is an error-tolerant version of a forward-backward-forward 
splitting algorithm originally proposed by Tseng 
\cite[Theorem~3.4(b)]{Tsen00}.

\begin{lemma}{\rm \cite[Theorem~2.5(i)--(ii)]{Siop11}}
\label{l:broome-st2010-10-27}
Let $\KKK$ be a real Hilbert space, let 
$\boldsymbol{P}\colon\KKK\to 2^{\KKK}$ be maximally monotone,
and let $\boldsymbol{Q}\colon\KKK\to\KKK$ be monotone and 
$\chi$-Lipschitzian for some $\chi\in\RPP$.
Suppose that $\zer(\boldsymbol{P}+\boldsymbol{Q})\neq\emp$. Let 
$(\boldsymbol{a}_n)_{n\in\NN}$, $(\boldsymbol{b}_n)_{n\in\NN}$,
and $(\boldsymbol{c}_n)_{n\in\NN}$ be absolutely summable 
sequences in $\KKK$, let $\boldsymbol{w}_0\in\KKK$, let 
$\varepsilon\in\left]0,1/(\chi+1)\right[$,
let $(\gamma_n)_{n\in\NN}$ be a sequence in 
$[\varepsilon,(1-\varepsilon)/\chi]$, and set
\begin{equation}
\label{e:rio2010-10-11u}
\begin{array}{l}
\text{For}\;n=0,1,\ldots\\
\left\lfloor
\begin{array}{l}
\boldsymbol{s}_n=\boldsymbol{w}_n-
\gamma_n(\boldsymbol{Q}\boldsymbol{w}_n+\boldsymbol{a}_n)\\
\boldsymbol{p}_n=J_{\gamma_n\boldsymbol{P}}\,\boldsymbol{s}_n
+\boldsymbol{b}_n\\
\boldsymbol{q}_n=\boldsymbol{p}_n-\gamma_n(\boldsymbol{Q}
\boldsymbol{p}_n+\boldsymbol{c}_n)\\
\boldsymbol{w}_{n+1}=\boldsymbol{w}_n-
\boldsymbol{s}_n+\boldsymbol{q}_n.
\end{array}
\right.\\[2mm]
\end{array}
\end{equation}
Then $\sum_{n\in\NN}\|\boldsymbol{w}_n-\boldsymbol{p}_n\|^2<\pinf$ 
and there exists 
$\boldsymbol{\overline{w}}\in\zer(\boldsymbol{P}+\boldsymbol{Q})$
such that $\boldsymbol{w}_n\weakly\boldsymbol{\overline{w}}$ and 
$\boldsymbol{p}_n\weakly\boldsymbol{\overline{w}}$. 
\end{lemma}

\begin{lemma}{\rm \cite[Proposition~23.15(ii) and 23.18]{Livre1}}
\label{l:4}
Let $\HH$ be a real Hilbert space, let $A\colon\HH\to 2^{\HH}$ 
be a maximally monotone operator, let $\gamma\in\RPP$, and let 
$x$ and $r$ be in $\HH$. Then $J_{\gamma(r+A^{-1})}x=x-\gamma
(r+J_{\gamma^{-1}A}(\gamma^{-1}x-r))$.
\end{lemma}

\begin{lemma}{\rm \cite[Proposition~2.8]{Siop11}}
\label{l:1}
Let $\HHH$ and $\GGG$ be two real Hilbert spaces, let 
$\boldsymbol{E}\colon\HHH\to 2^{\HHH}$ and 
$\boldsymbol{F}\colon\GGG\to 2^{\GGG}$ be maximally monotone, 
let $\boldsymbol{L}\in\BL(\HHH,\GGG)$, let 
$\boldsymbol{z}\in\HHH$, and let $\boldsymbol{r}\in\GGG$.
Set $\KKK=\HHH\oplus\GGG$,
\begin{equation}
\label{e:r8reXy-10-31z}
\begin{cases}
\boldsymbol{M}\colon\KKK\to 2^{\KKK}\colon
(\boldsymbol{x},\boldsymbol{v})\mapsto
(-\boldsymbol{z}+\boldsymbol{E}\boldsymbol{x})
\times(\boldsymbol{r}+\boldsymbol{F}^{-1}\boldsymbol{v})\\
\boldsymbol{S}\colon\KKK\to\KKK\colon(\boldsymbol{x},
\boldsymbol{v})\mapsto(\boldsymbol{L}^*\boldsymbol{v},
-\boldsymbol{L}\boldsymbol{x}),
\end{cases}
\end{equation}
and
\begin{equation}
\label{e:r8reXy-10-31Z}
\begin{cases}
\boldsymbol{\mathfrak P}=\menge{\boldsymbol{x}\in\HHH}
{\boldsymbol{z}\in\boldsymbol{E}\boldsymbol{x}+
\boldsymbol{L}^*(\boldsymbol{F}
(\boldsymbol{L}\boldsymbol{x}-\boldsymbol{r}))}\\
\boldsymbol{\mathfrak D}=\menge{\boldsymbol{v}\in\GGG}
{-\boldsymbol{r}\in -\boldsymbol{L}(\boldsymbol{E}^{-1}
(\boldsymbol{z}-\boldsymbol{L}^*\boldsymbol{v}))+
\boldsymbol{F}^{-1}\boldsymbol{v}}.
\end{cases}
\end{equation}
Then $\zer(\boldsymbol{M}+\boldsymbol{S})$ is a closed convex 
subset of $\boldsymbol{\mathfrak P}\times\boldsymbol{\mathfrak D}$,
and $\boldsymbol{\mathfrak P}\neq\emp$ $\Leftrightarrow$
$\zer(\boldsymbol{M}+\boldsymbol{S})\neq\emp$
$\Leftrightarrow$ $\boldsymbol{\mathfrak D}\neq\emp$.
\end{lemma}

The following theorem contains our algorithm for solving
Problem~\ref{prob:1} and states its main asymptotic properties.
In this primal-dual splitting algorithm, each single-valued 
operators is used explicitly, while each set-valued operators is 
activated via its resolvent. Approximations in the
evaluations of the operators are tolerated and modeled
by absolutely summable error sequences. The algorithm consists of
three main loops, each of which can be implemented on a parallel
architecture.

\begin{theorem}
\label{t:1}
Consider the setting of Problem~\ref{prob:1}.
For every $i\in\{1,\ldots,m\}$, let $(a_{1,i,n})_{n\in\NN}$,
$(b_{1,i,n})_{n\in\NN}$, and $(c_{1,i,n})_{n\in\NN}$ be 
absolutely summable sequences in $\HH_i$ and, for every 
$k\in\{1,\ldots,K\}$, let $(a_{2,k,n})_{n\in\NN}$, 
$(b_{2,k,n})_{n\in\NN}$, and $(c_{2,k,n})_{n\in\NN}$ be 
absolutely summable sequences in $\GG_k$. 
Let $x_{1,0}\in\HH_1$, \ldots, $x_{m,0}\in\HH_m$, 
$v_{1,0}\in\GG_1$, \ldots, $v_{K,0}\in\GG_K$, let
$\varepsilon\in\left]0,1/(\beta+1)\right[$, 
let $(\gamma_n)_{n\in\NN}$ be a sequence in 
$[\varepsilon,(1-\varepsilon)/\beta]$, and set
\begin{equation}
\label{e:5h8Njiq-11b}
\begin{array}{l}
\text{For}\;n=0,1,\ldots\\
\left\lfloor
\begin{array}{l}
\text{For}\;i=1,\ldots,m\\
\left\lfloor
\begin{array}{l}
s_{1,i,n}=x_{i,n}-\gamma_n\Big(C_ix_{i,n}+
\sum_{k=1}^KL_{ki}^*v_{k,n}+a_{1,i,n}\Big)\\[1mm]
p_{1,i,n}=J_{\gamma_n A_i}(s_{1,i,n}+\gamma_nz_i)+b_{1,i,n}\\[1mm]
\end{array}
\right.\\[1mm]
\text{For}\;k=1,\ldots,K\\
\left\lfloor
\begin{array}{l}
s_{2,k,n}=v_{k,n}-\gamma_n\Big(D_k^{-1}v_{k,n}-
\sum_{i=1}^mL_{ki}x_{i,n}+a_{2,k,n}\Big)\\[2mm]
p_{2,k,n}=s_{2,k,n}-\gamma_n\big(r_k+J_{\gamma_n^{-1}B_k}
(\gamma_n^{-1}s_{2,k,n}-r_k)+b_{2,k,n}\big)\\[2mm]
q_{2,k,n}=p_{2,k,n}-\gamma_n\Big(D_k^{-1}p_{2,k,n}-
\sum_{i=1}^mL_{ki}p_{1,i,n}+c_{2,k,n}\Big)\\
v_{k,n+1}=v_{k,n}-s_{2,k,n}+q_{2,k,n}
\end{array}
\right.\\[1mm]
\text{For}\;i=1,\ldots,m\\
\left\lfloor
\begin{array}{l}
q_{1,i,n}=p_{1,i,n}-\gamma_n\Big(C_ip_{1,i,n}+
\sum_{k=1}^KL_{ki}^*p_{2,k,n}+c_{1,i,n}\Big)\\
x_{i,n+1}=x_{i,n}-s_{1,i,n}+q_{1,i,n}.
\end{array}
\right.\\
\end{array}
\right.\\
\end{array}
\end{equation}
Then the following hold.
\begin{enumerate}
\item
\label{t:1i}
$(\forall i\in\{1,\ldots,m\})$
$\sum_{n\in\NN}\|x_{i,n}-p_{1,i,n}\|^2<\pinf$.
\item
\label{t:1i'}
$(\forall k\in\{1,\ldots,K\})$
$\sum_{n\in\NN}\|v_{k,n}-p_{2,k,n}\|^2<\pinf$.
\item
\label{t:1ii}
There exist a solution $(\overline{x_1},\ldots,\overline{x_m})$ 
to \eqref{e:IJ834hj8fr-24p} and a solution 
$(\overline{v_1},\ldots,\overline{v_K})$ to \eqref{e:IJ834hj8fr-24d} 
such that the following hold.
\begin{enumerate}
\item
\label{t:1iia}
$(\forall i\in\{1,\ldots,m\})$ 
$z_i-\sum_{k=1}^KL_{ki}^*\overline{v_k}\in
A_i\overline{x_i}+C_i\overline{x_i}$.
\item
\label{t:1iib}
$(\forall k\in\{1,\ldots,K\})$
$\sum_{i=1}^mL_{ki}\overline{x_i}-r_k\in B_k^{-1}\overline{v_k}+
D_k^{-1}\overline{v_k}$.
\item
\label{t:1iic}
$(\forall i\in\{1,\ldots,m\})$ $x_{i,n}\weakly\overline{x_i}~$ and
$~p_{1,i,n}\weakly\overline{x_i}$.
\item
\label{t:1iid}
$(\forall k\in\{1,\ldots,K\})$ $v_{k,n}\weakly\overline{v_k}~$ and 
$~p_{2,k,n}\weakly\overline{v_k}$.
\item 
\label{t:1iie}
Suppose that, for some $j\in\{1,\ldots,m\}$, $A_j$ or $C_j$ is 
uniformly monotone at $\overline{x_j}$. Then 
$x_{j,n}\to\overline{x_j}$ and $p_{1,j,n}\to\overline{x_j}$.
\item 
\label{t:1iif}
Suppose that, for some $l\in\{1,\ldots,K\}$, $B_l$ or $D_l$ is 
couniformly monotone at $\overline{v_l}$. Then
$v_{l,n}\to\overline{v_l}$ and $p_{2,l,n}\to\overline{v_l}$.
\end{enumerate}
\end{enumerate}
\end{theorem}
\begin{proof}
Let us introduce the Hilbert direct sums 
\begin{equation}
\label{e:IJ834hj8fr-30}
\HHH=\HH_1\oplus\cdots\oplus\HH_m,\quad
\GGG=\GG_1\oplus\cdots\oplus\GG_K,
\quad\text{and}\quad
\KKK=\HHH\oplus\GGG,
\end{equation}
and let us denote by $\boldsymbol{x}=(x_i)_{1\leq i\leq m}$ and 
$\boldsymbol{v}=(v_k)_{1\leq k\leq K}$ generic elements in 
$\HHH$ and $\GGG$, respectively. We also define 
\begin{equation}
\label{e:5h8Njiq-12b}
\begin{cases}
\boldsymbol{A}\colon\HHH\to 2^{\HHH}\colon\boldsymbol{x}\mapsto
\overset{m}{\underset{i=1}{\cart}}A_ix_i\\
\boldsymbol{C}\colon\HHH\to\HHH\colon\boldsymbol{x}\mapsto
(C_ix_i)_{1\leq i\leq m}\\
\boldsymbol{E}=\boldsymbol{A}+\boldsymbol{C}\\
\boldsymbol{L}\colon\HHH\to\GGG\colon\boldsymbol{x}\mapsto
\bigg(\Sum_{i=1}^mL_{ki}x_i\bigg)_{1\leq k\leq K}\\
\boldsymbol{z}=(z_i)_{1\leq i\leq m}\\
\end{cases}
\quad\text{and}\qquad
\begin{cases}
\boldsymbol{B}\colon\GGG\to 2^{\GGG}\colon\boldsymbol{v}\mapsto 
\overset{K}{\underset{k=1}{\cart}}B_kv_k\\[3mm]
\boldsymbol{D}\colon\GGG\to 2^{\GGG}\colon\boldsymbol{v}\mapsto 
\overset{K}{\underset{k=1}{\cart}}D_kv_k\\[3mm]
\boldsymbol{F}=\boldsymbol{B}\infconv\boldsymbol{D}\\
\boldsymbol{r}=(r_k)_{1\leq k\leq K}.
\end{cases}
\end{equation}
It follows from \cite[Proposition~20.22 and 20.23, 
Corollaries~20.25 and 24.4(i)]{Livre1} that 
$\boldsymbol{A}$, $\boldsymbol{B}$, 
$\boldsymbol{C}$, $\boldsymbol{D}$, $\boldsymbol{E}$, and
$\boldsymbol{F}$ are maximally monotone. 
Moreover, $\boldsymbol{L}\in\BL(\HHH,\GGG)$,
$\boldsymbol{L}^*\colon\GGG\to\HHH\colon\boldsymbol{v}
\mapsto(\sum_{k=1}^KL_{ki}^*v_k)_{1\leq i\leq m}$, and
\begin{equation}
\label{e:5h8Njiq-12x}
(\forall\boldsymbol{x}\in\HHH)\quad
\|\boldsymbol{L}\boldsymbol{x}\|^2
=\sum_{k=1}^K\bigg\|\sum_{i=1}^mL_{ki}x_i\bigg\|^2
\leq\lambda\|\boldsymbol{x}\|^2.
\end{equation}
Next, we set
\begin{equation}
\label{e:5h8Njiq-12c}
\begin{cases}
\boldsymbol{M}\colon\KKK\to 2^{\KKK}\colon
(\boldsymbol{x},\boldsymbol{v})\mapsto
(-\boldsymbol{z}+\boldsymbol{E}\boldsymbol{x})
\times(\boldsymbol{r}+\boldsymbol{F}^{-1}\boldsymbol{v})\\
\boldsymbol{P}\colon\KKK\to 2^{\KKK}\colon
(\boldsymbol{x},\boldsymbol{v})\mapsto
(-\boldsymbol{z}+\boldsymbol{A}\boldsymbol{x})\times(\boldsymbol{r}
+\boldsymbol{B}^{-1}\boldsymbol{v})\\
\boldsymbol{Q}\colon\KKK\to\KKK\colon
(\boldsymbol{x},\boldsymbol{v})\mapsto
\big(\boldsymbol{C}\boldsymbol{x}+\boldsymbol{L}^*\boldsymbol{v},
\boldsymbol{D}^{-1}\boldsymbol{v}-\boldsymbol{L}
\boldsymbol{x}\big)\\
\boldsymbol{R}\colon\KKK\to\KKK\colon(\boldsymbol{x},
\boldsymbol{v})\mapsto(\boldsymbol{C}\boldsymbol{x},
\boldsymbol{D}^{-1}\boldsymbol{v})\\
\boldsymbol{S}\colon\KKK\to\KKK\colon(\boldsymbol{x},
\boldsymbol{v})\mapsto(\boldsymbol{L}^*\boldsymbol{v},
-\boldsymbol{L}\boldsymbol{x}).
\end{cases}
\end{equation}
Note that
\begin{equation}
\label{e:5h8Njiq-13a}
\zer(\boldsymbol{P}+\boldsymbol{Q})=
\menge{(\boldsymbol{x},\boldsymbol{v})\in\HHH\oplus\GGG}
{\boldsymbol{z}-\boldsymbol{L}^*\boldsymbol{v}\in
\boldsymbol{A}\boldsymbol{x}+\boldsymbol{C}\boldsymbol{x}
\quad\text{and}\quad\boldsymbol{L}\boldsymbol{x}-\boldsymbol{r}
\in\boldsymbol{B}^{-1}\boldsymbol{v}+
\boldsymbol{D}^{-1}\boldsymbol{v}}.
\end{equation}
Furthermore, in view of 
\cite[Propositions~20.22 and 20.23]{Livre1}, $\boldsymbol{P}$ 
is maximally monotone, and Lemma~\ref{l:4} and 
\cite[Proposition~23.16]{Livre1} yield
\begin{multline}
\label{e:5h8Njiq-08a}
(\forall\gamma\in\RPP)(\forall\boldsymbol{x}\in\HHH)
(\forall\boldsymbol{v}\in\GGG)\quad
J_{\gamma\boldsymbol{P}}(\boldsymbol{x},\boldsymbol{v})=
\Big(J_{\gamma A_1}(x_1+\gamma z_1),\ldots,
J_{\gamma A_m}(x_m+\gamma z_m),\\
v_1-\gamma\big(r_1+J_{\gamma^{-1}B_1}(\gamma^{-1}v_1-r_1)\big),
\ldots,v_K-\gamma\big(r_K+J_{\gamma^{-1}B_K}
(\gamma^{-1}v_K-r_K)\big)\Big).
\end{multline}
On the other hand, since $\boldsymbol{C}$ and $\boldsymbol{D}^{-1}$ 
are monotone and Lipschitzian with, respectively, constants 
$\mu=\text{max}_{1\leq i\leq m}\mu_i$ and
$\nu=\text{max}_{1\leq k\leq K}\nu_k$, 
$\boldsymbol{R}$ is monotone and Lipschitzian with
constant $\text{max}\{\mu,\nu\}$. In addition, it follows from 
\cite[Proposition~2.7(ii)]{Siop11} and \eqref{e:5h8Njiq-12x} 
that $\boldsymbol{S}\in\BL(\KKK,\KKK)$ is 
a skew (hence monotone) operator with 
$\|\boldsymbol{S}\|=\|\boldsymbol{L}\|\leq\sqrt{\lambda}$.
Altogether, since $\boldsymbol{Q}=\boldsymbol{R}+\boldsymbol{S}$,
we derive from \eqref{e:5h8Njiq-11a} that
\begin{equation}
\label{e:5h8Njiq-12f}
\boldsymbol{P}\;\text{is maximally monotone and}\;
\boldsymbol{Q}\;\text{is monotone and $\beta$-Lipschitzian.}
\end{equation}
Let us call $\boldsymbol{\mathfrak P}$ and 
$\boldsymbol{\mathfrak D}$ the sets of solutions to 
\eqref{e:IJ834hj8fr-24p} and \eqref{e:IJ834hj8fr-24d}, respectively.
It follows from \eqref{e:5h8Njiq-12b} that
\begin{equation}
\label{e:5h8Njiq-12d}
\begin{cases}
\boldsymbol{\mathfrak P}=\menge{\boldsymbol{x}\in\HHH}
{\boldsymbol{z}\in\boldsymbol{E}\boldsymbol{x}+
\boldsymbol{L}^*(\boldsymbol{F}
(\boldsymbol{L}\boldsymbol{x}-\boldsymbol{r}))}\\
\boldsymbol{\mathfrak D}=\menge{\boldsymbol{v}\in\GGG}
{-\boldsymbol{r}\in -\boldsymbol{L}(\boldsymbol{E}^{-1}
(\boldsymbol{z}-\boldsymbol{L}^*\boldsymbol{v}))+
\boldsymbol{F}^{-1}\boldsymbol{v}}.
\end{cases}
\end{equation}
Hence, since $\boldsymbol{\mathfrak P}\neq\emp$ by assumption, 
we deduce from Lemma~\ref{l:1} that 
\begin{equation}
\label{e:5h8Njiq-11g}
\emp\neq\zer(\boldsymbol{M}+\boldsymbol{S})
=\zer(\boldsymbol{P}+\boldsymbol{Q})
\subset\boldsymbol{\mathfrak P}\times\boldsymbol{\mathfrak D}.
\end{equation}
Thus, to solve Problem~\ref{prob:1}, it is enough to find a zero of
$\boldsymbol{P}+\boldsymbol{Q}$. For every $n\in\NN$, let us set
\begin{equation}
\label{e:5h8Njiq-12j}
\begin{cases}
\boldsymbol{w}_n=(x_{1,n},\ldots,x_{m,n},v_{1,n},\ldots,v_{K,n})\\
\boldsymbol{s}_n=(s_{1,1,n},\ldots,s_{1,m,n},s_{2,1,n},\ldots,
s_{2,K,n})\\
\boldsymbol{p}_n=(p_{1,1,n},\ldots,p_{1,m,n},p_{2,1,n},\ldots,
p_{2,K,n})\\
\boldsymbol{q}_n=(q_{1,1,n},\ldots,q_{1,m,n},q_{2,1,n},\ldots,
q_{2,K,n})
\end{cases}
\end{equation}
and
\begin{equation}
\label{e:5h8Njiq-12k}
\begin{cases}
\boldsymbol{a}_n=(a_{1,1,n},\ldots,a_{1,m,n},a_{2,1,n},\ldots,
a_{2,K,n})\\
\boldsymbol{b}_n=(b_{1,1,n},\ldots,b_{1,m,n},
-\gamma_nb_{2,1,n},\ldots,-\gamma_nb_{2,K,n})\\
\boldsymbol{c}_n=(c_{1,1,n},\ldots,c_{1,m,n},c_{2,1,n},\ldots,
c_{2,K,n}).
\end{cases}
\end{equation}
Then, using \eqref{e:5h8Njiq-12b}, \eqref{e:5h8Njiq-12c}, and 
\eqref{e:5h8Njiq-08a}, we see that \eqref{e:5h8Njiq-11b} reduces 
to \eqref{e:rio2010-10-11u}.
Moreover, our assumptions and \eqref{e:IJ834hj8fr-30} imply that
$(\boldsymbol{a}_n)_{n\in\NN}$, $(\boldsymbol{b}_n)_{n\in\NN}$,
and $(\boldsymbol{c}_n)_{n\in\NN}$ are absolutely summable 
sequences in $\KKK$. Hence, it follows from \eqref{e:5h8Njiq-12f}, 
\eqref{e:5h8Njiq-11g}, and Lemma~\ref{l:broome-st2010-10-27} that 
$\sum_{n\in\NN}\|\boldsymbol{w}_n-\boldsymbol{p}_n\|^2<\pinf$ and
that there exists 
$\boldsymbol{\overline{w}}\in\zer(\boldsymbol{P}+\boldsymbol{Q})$
such that $\boldsymbol{w}_n\weakly\boldsymbol{\overline{w}}$ and 
$\boldsymbol{p}_n\weakly\boldsymbol{\overline{w}}$. 
Upon setting $\boldsymbol{\overline{w}}=(\overline{x_1},\ldots,
\overline{x_m},\overline{v_1},\ldots,\overline{v_K})$ and 
appealing to \eqref{e:IJ834hj8fr-30} and \eqref{e:5h8Njiq-13a}, we 
thus obtain assertions \ref{t:1i}, \ref{t:1i'}, and 
\ref{t:1iia}--\ref{t:1iid}. 

\ref{t:1iie}: Let us introduce the variables
\begin{equation}
\label{e:5h8Njiq-19a}
(\forall i\in\{1,\ldots,m\})(\forall n\in\NN)\quad 
\begin{cases}
\widetilde{s}_{1,i,n}=x_{i,n}-\gamma_n\bigg(C_ix_{i,n}+
\Sum_{k=1}^KL_{ki}^*v_{k,n}\bigg)\\
\widetilde{p}_{1,i,n}=J_{\gamma_n A_i}(\widetilde{s}_{1,i,n}
+\gamma_nz_i)
\end{cases}
\end{equation}
and
\begin{equation}
\label{e:ford}
(\forall k\in\{1,\ldots,K\})(\forall n\in\NN)\quad 
\begin{cases}
\widetilde{s}_{2,k,n}=v_{k,n}-\gamma_n\bigg(D_k^{-1}v_{k,n}-
\Sum_{i=1}^mL_{ki}x_{i,n}\bigg)\\
\widetilde{p}_{2,k,n}=\widetilde{s}_{2,k,n}-\gamma_n\Big(r_k+
J_{\gamma_n^{-1}B_k}(\gamma_n^{-1}\widetilde{s}_{2,k,n}-r_k)\Big).
\end{cases}
\end{equation}
It follows from \eqref{e:5h8Njiq-11b} that
\begin{equation}
(\forall i\in\{1,\ldots,m\})(\forall n\in\NN)\quad
\|s_{1,i,n}-\widetilde{s}_{1,i,n}\|=\gamma_n\|a_{1,i,n}\|\leq
\beta^{-1}\|a_{1,i,n}\|.
\end{equation}
Hence, by virtue of the nonexpansiveness of the resolvents 
\cite[Proposition~23.7]{Livre1}, we have
\begin{align}
(\forall i\in\{1,\ldots,m\})(\forall n\in\NN)\;
\|p_{1,i,n}-\widetilde{p}_{1,i,n}\| 
&=\|J_{\gamma_n A_i}(s_{1,i,n}+\gamma_nz_i)+b_{1,i,n}\!-\!
J_{\gamma_n A_i}(\widetilde{s}_{1,i,n}+\gamma_nz_i)\|\nonumber\\
&\leq\|s_{1,i,n}-\widetilde{s}_{1,i,n}\|+\|b_{1,i,n}\|\nonumber\\
&\leq\beta^{-1}\|a_{1,i,n}\|+\|b_{1,i,n}\|.
\end{align}
In turn, since, for every $i\in\{1,\ldots,m\}$,
$(a_{1,i,n})_{n\in\NN}$ and $(b_{1,i,n})_{n\in\NN}$ are 
absolutely summable, we get
\begin{equation}
\label{e:convyp1t}
(\forall i\in\{1,\ldots,m\})\quad
s_{1,i,n}-\widetilde{s}_{1,i,n}\to 0 
\quad\text{and}\quad
p_{1,i,n}-\widetilde{p}_{1,i,n}\to 0.
\end{equation}
Likewise, we derive from \eqref{e:5h8Njiq-11b} and \eqref{e:ford} 
that
\begin{equation}
\label{e:convyp2t}
(\forall k\in\{1,\ldots,K\})\quad
s_{2,k,n}-\widetilde{s}_{2,k,n}\to 0 \quad\text{and}\quad
p_{2,k,n}-\widetilde{p}_{2,k,n}\to 0.
\end{equation}
On the other hand, we deduce from \ref{t:1iia} that
\begin{equation}
\label{e:barletta}
(\forall i\in\{1,\ldots,m\})(\exi u_i\in\HH_i)
\quad u_i\in A_i\overline{x_i}\quad\text{and}\quad
z_i=u_i+\sum_{k=1}^KL_{ki}^*\overline{v_k}+C_i\overline{x_i},
\end{equation}
and from \ref{t:1iib} that
\begin{equation}
\label{e:lecce}
(\forall k\in\{1,\ldots,K\})\quad\overline{v_k}\in
B_k\bigg(\sum_{i=1}^mL_{ki}\overline{x_i}-r_k-
D_k^{-1}\overline{v_k}\bigg).
\end{equation}
In addition, \eqref{e:5h8Njiq-19a} yields
\begin{equation}
\label{e:bari}
(\forall i\in\{1,\ldots,m\})(\forall n\in\NN)\quad 
\frac{x_{i,n}-\widetilde{p}_{1,i,n}}{\gamma_n}
-\sum_{k=1}^KL_{ki}^*v_{k,n}-C_ix_{i,n}
+z_i\in A_i\widetilde{p}_{1,i,n},
\end{equation}
while \eqref{e:ford} yields
\begin{equation}
\label{e:bari2}
(\forall k\in\{1,\ldots,K\})(\forall n\in\NN)\quad 
\widetilde{p}_{2,k,n}\in B_k\bigg(
\frac{v_{k,n}-\widetilde{p}_{2,k,n}}{\gamma_n}
+\sum_{i=1}^mL_{ki}x_{i,n}
-r_k-D_k^{-1}v_{k,n}\bigg).
\end{equation}
Now, let us set
\begin{multline}
\label{e:defalpha12}
(\forall n\in\NN)\quad 
\delta_n=\sum_{k=1}^K\bigg(\frac{1}{\varepsilon}+\nu_k\bigg)
\|v_{k,n}-\widetilde{p}_{2,k,n}\|\,
\|\widetilde{p}_{2,k,n}-\overline{v_k}\|
\quad\text{and}\quad
(\forall i\in\{1,\ldots,m\}) \\
\alpha_{i,n}=\|\widetilde{p}_{1,i,n}-x_{i,n}\|
\bigg(\frac{1}{\varepsilon}\|\widetilde{p}_{1,i,n}-\overline{x_i}\|
+\mu_i\|x_{i,n}-\overline{x_i}\|+\sum_{k=1}^K\|L_{ki}\|\,
\|v_{k,n}-\overline{v_k}\|\bigg).
\end{multline}
It follows from \ref{t:1i}, \ref{t:1i'}, \ref{t:1iic}, 
\ref{t:1iid}, \eqref{e:convyp1t}, and \eqref{e:convyp2t} that
\begin{equation}
\label{e:2011-07-09a}
\delta_n\to 0\quad\text{and}\quad(\forall i\in\{1,\ldots,m\})
\quad\alpha_{i,n}\to 0.
\end{equation}
Using the Cauchy-Schwarz inequality, the Lipschitz-continuity and 
the monotonicity of the operators $(C_i)_{1\leq i\leq m}$, 
\eqref{e:barletta}, \eqref{e:bari}, and the monotonicity of the 
operators $(A_i)_{1\leq i\leq m}$, we obtain
\[
(\forall i\in\{1,\ldots,m\})(\forall n\in\NN)\quad
\alpha_{i,n}+\Scal{x_{i,n}-\overline{x_i}}{\sum_{k=1}^K
L_{ki}^*(\overline{v_k}-v_{k,n})}
\]\\[-14mm]
\begin{align}
&\geq\|\widetilde{p}_{1,i,n}-x_{i,n}\|
\big(\varepsilon^{-1}\|\widetilde{p}_{1,i,n}-\overline{x_i}\|+
\|C_i x_{i,n}-C_i\overline{x_i}\|\big)+
\Scal{\widetilde{p}_{1,i,n}-x_{i,n}}
{\sum_{k=1}^KL_{ki}^*(\overline{v_k}-v_{k,n})}\nonumber\\[-3mm]
&\quad\;+\Scal{x_{i,n}-\overline{x_i}}{\sum_{k=1}^KL_{ki}^*
(\overline{v_k}-v_{k,n})}\nonumber\\[-3mm]
&=\|\widetilde{p}_{1,i,n}-x_{i,n}\|
\big(\varepsilon^{-1} \|\widetilde{p}_{1,i,n}-\overline{x_i}\|
+\|C_i x_{i,n}-C_i \overline{x_i}\|\big)+
\Scal{\widetilde{p}_{1,i,n}-\overline{x_i}}{\sum_{k=1}^K
L_{ki}^*(\overline{v_k}-v_{k,n})}\nonumber\\
&\geq\Scal{\widetilde{p}_{1,i,n}-\overline{x_i}}
{\frac{x_{i,n}-\widetilde{p}_{1,i,n}}{\gamma_n}+\sum_{k=1}^K
L_{ki}^*(\overline{v_k}-v_{k,n})}+
\scal{\widetilde{p}_{1,i,n}-x_{i,n}}{C_i
\overline{x_i}-C_i x_{i,n}}\nonumber\\
&=\Scal{\widetilde{p}_{1,i,n}-\overline{x_i}}
{\frac{x_{i,n}-\widetilde{p}_{1,i,n}}{\gamma_n}-
\sum_{k=1}^K L_{ki}^* v_{k,n}-C_ix_{i,n}+
\sum_{k=1}^KL_{ki}^*\overline{v_k}+ C_i\overline{x_i}}
\nonumber\\
&\quad\;+\scal{x_{i,n}-\overline{x_i}}
{C_ix_{i,n}-C_i\overline{x_i} }\nonumber\\
&=\Scal{\widetilde{p}_{1,i,n}-\overline{x_i}}{
\frac{x_{i,n}-\widetilde{p}_{1,i,n}}{\gamma_n}
-\sum_{k=1}^K L_{ki}^*v_{k,n}-C_i x_{i,n}+z_i-u_i}\nonumber\\
&\quad\;+\scal{x_{i,n}-\overline{x_i}}
{C_ix_{i,n}-C_i\overline{x_i}}\label{e:HdN1}\\
&\geq\Scal{\widetilde{p}_{1,i,n}-\overline{x_i}}
{\bigg(\frac{x_{i,n}-\widetilde{p}_{1,i,n}}{\gamma_n}-
\sum_{k=1}^KL_{ki}^*v_{k,n}-C_i x_{i,n}+z_i\bigg)-u_i}
\label{e:HdN2}\\
&\geq 0.
\label{e:taranto1}
\end{align}
On the other hand, since the operators $(D_k^{-1})_{1\leq k\leq K}$
are Lipschitzian and monotone, and since the operators 
$(B_k)_{1\leq k\leq K}$ are monotone, we deduce from 
\eqref{e:defalpha12}, \eqref{e:lecce}, and \eqref{e:bari2} that 
\[
(\forall l\in\{1.\ldots,K\})(\forall n\in\NN)\quad
\delta_n+\sum_{i=1}^m\Scal{x_{i,n}-\overline{x_i}}{\sum_{k=1}^K
L_{ki}^*(\widetilde{p}_{2,k,n}-\overline{v_k})}
\]\\[-9mm]
\begin{align}
&\geq\sum_{k=1}^K\Scal{\frac{v_{k,n}-\widetilde{p}_{2,k,n}}
{\gamma_n}+D_k^{-1}\widetilde{p}_{2,k,n}-
D_k^{-1}v_{k,n}+\sum_{i=1}^mL_{ki}
(x_{i,n}-\overline{x_i})}
{\widetilde{p}_{2,k,n}-\overline{v_k}}\nonumber\\
&=\sum_{k=1}^K\Scal{\bigg(\frac{v_{k,n}\!-\!
\widetilde{p}_{2,k,n}}{\gamma_n}+\sum_{i=1}^mL_{ki}x_{i,n}
\!-\!r_k\!-\!D_k^{-1}v_{k,n}\bigg)\!-\!\bigg(\sum_{i=1}^mL_{ki}
\overline{x_i}\!-\!r_k\!-\!D_k^{-1}\overline{v_k}\bigg)}
{\widetilde{p}_{2,k,n}\!-\!\overline{v_k}}\nonumber\\
&\quad\;+\sum_{k=1}^K\scal{D_k^{-1}\widetilde{p}_{2,k,n}
-D_k^{-1}\overline{v_k}}
{\widetilde{p}_{2,k,n}-\overline{v_k}}\label{e:HdN4}\\
&\geq\Scal{\bigg(\frac{v_{l,n}-\widetilde{p}_{2,l,n}}{\gamma_n}+
\sum_{i=1}^mL_{li}x_{i,n}-r_l-D_l^{-1}v_{l,n}\bigg)-
\bigg(\sum_{i=1}^mL_{li}\overline{x_i}-r_l-D_l^{-1}
\overline{v_l}\bigg)}{\widetilde{p}_{2,l,n}-
\overline{v_l}}\nonumber\\
&\quad\;+\scal{D_l^{-1}\widetilde{p}_{2,l,n}
-D_l^{-1}\overline{v_l}}
{\widetilde{p}_{2,l,n}-\overline{v_l}}\label{e:HdN5}\\
&\geq\Scal{\bigg(\frac{v_{l,n}-\widetilde{p}_{2,l,n}}{\gamma_n}+
\sum_{i=1}^mL_{li}x_{i,n}-r_l-D_l^{-1}v_{l,n}\bigg)-
\bigg(\sum_{i=1}^mL_{li}\overline{x_i}-r_l-D_l^{-1}
\overline{v_l}\bigg)}{\widetilde{p}_{2,l,n}-
\overline{v_l}}\label{e:HdN9}\\
&\geq 0.
\label{e:taranto3}
\end{align}
We consider two cases.
\begin{itemize}
\item
If $A_j$ is uniformly monotone at $\overline{x_j}$, then, in view
of \eqref{e:HdN2}, \eqref{e:barletta}, \eqref{e:bari}, and 
\eqref{e:genova}, there exists an increasing function 
$\phi_{A_j}\colon\RP\to\RPX$ that vanishes only at $0$ such that
\begin{equation}
\label{e:taranto2}
(\forall n\in\NN)\quad
\alpha_{j,n}+\Scal{x_{j,n}-\overline{x_j}}{\sum_{k=1}^K
L_{kj}^*(\overline{v_k}-v_{k,n})}\geq\phi_{A_j}
(\|\widetilde{p}_{1,j,n}-\overline{x_j}\|).
\end{equation}
Combining \eqref{e:taranto3}, \eqref{e:taranto1}, and 
\eqref{e:taranto2} yields
\begin{equation}
\label{e:conclustrong}
(\forall n\in\NN)\quad
\delta_n+\sum_{i=1}^m\alpha_{i,n}+
\sum_{i=1}^m\Scal{x_{i,n}-\overline{x_i}}
{\sum_{k=1}^KL_{ki}^*(\widetilde{p}_{2,k,n}-v_{k,n})}
\geq\phi_{A_j}(\|\widetilde{p}_{1,j,n}-\overline{x_j}\|).
\end{equation}
It follows from \eqref{e:2011-07-09a}, \ref{t:1i'}, \ref{t:1iic},
\eqref{e:convyp2t}, and \cite[Lemma~2.41(iii)]{Livre1} 
that $\phi_{A_j}(\|\widetilde{p}_{1,j,n}-\overline{x_j}\|)\to 0$ 
and, in turn, that $\widetilde{p}_{1,j,n}\to\overline{x_j}$.
In view of \ref{t:1i} and \eqref{e:convyp1t}, we get
$p_{1,j,n}\to\overline{x_j}$ and $x_{j,n}\to\overline{x_j}$.
\item
If $C_j$ is uniformly monotone at $\overline{x_j}$, then we derive 
from \eqref{e:taranto3}, \eqref{e:HdN1}, and 
\eqref{e:taranto1} that there exists an increasing function 
$\phi_{C_j}\colon\RP\to\RPX$ that vanishes 
only at $0$ such that 
\begin{multline}
(\forall n\in\NN)\quad
\delta_n+\sum_{i=1}^m\alpha_{i,n}+\sum_{i=1}^m\Scal{x_{i,n}-
\overline{x_i}}{\sum_{k=1}^KL_{ki}^*(\widetilde{p}_{2,k,n}-
v_{k,n})}\\
\geq\phi_{C_j}(\|x_{j,n}-\overline{x_j}\|).
\end{multline}
This implies that $\phi_{C_j}(\|x_{j,n}-\overline{x_j}\|)\to 0$
and hence that $x_{j,n}\to\overline{x_j}$. Finally, \ref{t:1i}
yields $p_{1,j,n}\to\overline{x_j}$.
\end{itemize}

\ref{t:1iif}:
We consider two cases.
\begin{itemize}
\item
If $B_l$ is couniformly monotone at $\overline{v_l}$, then
\eqref{e:HdN9}, \eqref{e:lecce}, and 
\eqref{e:bari2} imply that there exists an 
increasing function $\phi_{B_l^{-1}}\colon\RP\to\RPX$ that 
vanishes only at $0$ such that
\[
(\forall n\in\NN)\quad
\delta_n+\sum_{i=1}^m\Scal{x_{i,n}-\overline{x_i}}{\sum_{k=1}^K
L_{ki}^*(\widetilde{p}_{2,k,n}-\overline{v_k})}
\]\\[-9mm]
\begin{align}
&\geq\Scal{\bigg(\frac{v_{l,n}\!-\!\widetilde{p}_{2,l,n}}{\gamma_n}+
\sum_{i=1}^mL_{li}x_{i,n}\!-\!r_l-D_l^{-1}v_{l,n}\bigg)-
\bigg(\sum_{i=1}^mL_{li}\overline{x_i}\!-\!r_l-D_l^{-1}
\overline{v_l}\bigg)}{\widetilde{p}_{2,l,n}\!-\!
\overline{v_l}}\nonumber\\
&\geq\phi_{B_l^{-1}}(\|\widetilde{p}_{2,l,n}-\overline{v_l}\|).
\end{align}
Combining this with \eqref{e:taranto1} yields
\begin{equation}
(\forall n\in\NN)\quad
\delta_n+\sum_{i=1}^m\alpha_{i,n}+
\sum_{i=1}^m\Scal{x_{i,n}-\overline{x_i}}{\sum_{k=1}^K
L_{ki}^*(\widetilde{p}_{2,k,n}-v_{k,n})}
\geq\phi_{B_l^{-1}}(\|\widetilde{p}_{2,l,n}-\overline{v_l}\|).
\end{equation}
Hence, using \eqref{e:2011-07-09a}, \ref{t:1i'}, 
\ref{t:1iic}, \eqref{e:convyp2t}, and 
\cite[Lemma~2.41(iii)]{Livre1}, we get
$\phi_{B_l^{-1}}(\|\widetilde{p}_{2,l,n}-\overline{v_l}\|)\to 0$ 
and, in turn, $\widetilde{p}_{2,l,n}\to\overline{v_l}$. 
Using to \eqref{e:convyp2t} and \ref{t:1i'}, we conclude
that $p_{2,l,n}\to\overline{v_l}$ and $v_{l,n}\to\overline{v_l}$. 
\item
If $D_l$ is couniformly monotone at $\overline{v_l}$, then 
it follows from \eqref{e:HdN5} and \eqref{e:taranto3} that 
there exists an increasing function 
$\phi_{D_l^{-1}}\colon\RP\to\RPX$ that vanishes only at 
$0$ such that
\begin{align}
(\forall n\in\NN)\quad
\delta_n+\sum_{i=1}^m\Scal{x_{i,n}-\overline{x_i}}{\sum_{k=1}^K
L_{ki}^*(\widetilde{p}_{2,k,n}-\overline{v_k})}
&\geq\scal{D_l^{-1}\widetilde{p}_{2,l,n}
-D_l^{-1}\overline{v_l}}
{\widetilde{p}_{2,l,n}-\overline{v_l}}\nonumber\\
&\geq\phi_{D_l^{-1}}(\|\widetilde{p}_{2,l,n}-\overline{v_l}\|).
\end{align}
Thus, \eqref{e:taranto1} yields
\begin{equation}
(\forall n\in\NN)\quad
\delta_n+\sum_{i=1}^m\alpha_{i,n}+
\sum_{i=1}^m\Scal{x_{i,n}-\overline{x_i}}{\sum_{k=1}^K
L_{ki}^*(\widetilde{p}_{2,k,n}-v_{k,n})}
\geq\phi_{D_l^{-1}}(\|\widetilde{p}_{2,l,n}-\overline{v_l}\|),
\end{equation}
and we conclude as above.
\end{itemize}
\end{proof}

\begin{remark}
When $m=1$, Theorem~\ref{t:1} specializes to
\cite[Theorem~3.1]{Svva12}. Our proof of 
Theorem~\ref{t:1}\ref{t:1i}--\ref{t:1iid} hinges on
a self-contained application of Lemmas~\ref{l:1} and 
\ref{l:broome-st2010-10-27} in the primal-dual product space 
$\KKK$ of \eqref{e:IJ834hj8fr-30}. Alternatively, these results could
be obtained as an application of \cite[Theorem~3.1]{Svva12}
using the product space $\HHH$ of \eqref{e:IJ834hj8fr-30} as a 
primal space. This strategy, however, would not allow us to 
recover the strong convergence results of 
Theorem~\ref{t:1}\ref{t:1iie}.
\end{remark}

\begin{remark}
It follows from the Cauchy-Schwarz inequality that,
for every $(x_i)_{1\leq i\leq m}\in\bigoplus_{i=1}^m\HH_i$, 
\begin{equation}
\label{e:2012-10-12xu}
\sum_{k=1}^K\bigg\|\sum_{i=1}^mL_{ki}x_i\bigg\|^2
\leq\sum_{k=1}^K\bigg(\sum_{i=1}^m\|L_{ki}\|\,\|x_i\|\bigg)^2
\leq\sum_{k=1}^K\bigg(\sum_{i=1}^m\|L_{ki}\|^2\bigg)\,
\bigg(\sum_{i=1}^m\|x_i\|^2\bigg).
\end{equation}
Hence, in general, one can use
$\lambda=\sum_{k=1}^K\sum_{i=1}^m\|L_{ki}\|^2$ in 
\eqref{e:5h8Njiq-11a}. However, as will be seen in 
subsequent sections, this bound can be improved when the 
operator $\boldsymbol{L}$ of \eqref{e:5h8Njiq-12b} has a 
special structure.
\end{remark}

In the remainder the paper, we highlight a few instantiations of 
Theorem~\ref{t:1} that illustrate the variety of problems to 
which it can be applied and which are not explicitly solvable via 
existing techniques. 

\section{Inclusions involving general parallel sums}
\label{sec:4}

The first special case of Problem~\ref{prob:1} we feature is an 
extension of a univariate inclusion problem investigated in 
\cite{Svva12}, which involves parallel sums with monotone 
operators admitting Lipschitzian inverses. In the following 
formulation, we lift this restriction.

\begin{problem}
\label{prob:2}
Let $\HH$ be a real Hilbert space, let $K_1$, $K_2$, and $K$ be 
integers such that $0\leq K_1\leq K_2\leq K\geq 1$, let $z\in\HH$, 
let $A\colon\HH\to 2^{\HH}$ be maximally monotone, and let 
$C\colon\HH\to\HH$ be monotone and $\mu$-Lipschitzian for some 
$\mu\in\RP$. For every integer $k\in\{1,\ldots,K\}$, let $\GG_k$ 
be a real Hilbert space, let $r_k\in\GG_k$, let 
$B_k\colon\GG_k\to 2^{\GG_k}$ and $S_k\colon\GG_k\to 2^{\GG_k}$ 
be maximally monotone, and let $L_k\in\BL(\HH,\GG_k)$; 
moreover, if $K_1+1\leq k\leq K_2$, $S_k\colon\GG_k\to{\GG_k}$ 
is $\beta_k$-Lipschitzian for some $\beta_k\in\RP$, 
and, if $K_2+1\leq k\leq K$, $S_k^{-1}\colon\GG_k\to\GG_k$ is 
$\beta_k$-Lipschitzian for some $\beta_k\in\RP$. It is assumed that
\begin{equation}
\label{e:uH98j3h-07c}
\beta=\text{max}\big\{\mu,\beta_{K_1+1},\ldots,\beta_K\big\}+
\sqrt{1+\sum_{k=1}^K\|L_k\|^2}>0,
\end{equation}
and that the inclusion
\begin{equation}
\label{e:uH98j3h-07p}
\text{find}\;\;\overline{x}\in\HH\;\;\text{such that}\;\;
z\in A\overline{x}+\sum_{k=1}^KL_k^*\big((B_k\infconv S_k)
(L_k\overline{x}-r_k)\big)+C\overline{x}
\end{equation}
possesses at least one solution. Solve \eqref{e:uH98j3h-07p} 
together with the dual problem
\begin{multline}
\label{e:uH98j3h-07d}
\text{find}\;\;\overline{v_1}\in\GG_1,\:\ldots,\:
\overline{v_K}\in\GG_K\;\:\text{such that}\\
(\forall k\in\{1,\ldots,K\})\quad
-r_k\in -L_k\bigg((A+C)^{-1}\bigg(z-\Sum_{l=1}^K
L_l^*\overline{v_l}\bigg)\bigg)+B_k^{-1}\overline{v_k}+
S_k^{-1}\overline{v_k}.
\end{multline}
\end{problem}

\begin{proposition}
\label{p:1}
Consider the setting of Problem~\ref{prob:2}.
Let $(a_{1,1,n})_{n\in\NN}$, $(b_{1,1,n})_{n\in\NN}$, and 
$(c_{1,1,n})_{n\in\NN}$ be absolutely summable sequences in $\HH$.
For every integer $k\in\{1,\ldots,K\}$, let 
$(a_{2,k,n})_{n\in\NN}$, $(b_{2,k,n})_{n\in\NN}$, and 
$(c_{2,k,n})_{n\in\NN}$ be absolutely summable sequences in 
$\GG_k$; moreover, if $1\leq k\leq K_1$, let 
$(b_{1,k+1,n})_{n\in\NN}$ be 
an absolutely summable sequence in $\GG_k$, and, if
$K_1+1\leq k\leq K_2$\,, let $(a_{1,k+1,n})_{n\in\NN}$ 
and $(c_{1,k+1,n})_{n\in\NN}$ be absolutely summable sequences in 
$\GG_k$. Let $x_{0}\in\HH$, $y_{1,0}\in\GG_1$, 
\ldots, $y_{K_2,0}\in\GG_{K_2}$, $v_{1,0}\in\GG_1$, \ldots, and
$v_{K,0}\in\GG_K$, let $\varepsilon\in\left]0,1/(\beta+1)\right[$,
let $(\gamma_n)_{n\in\NN}$ be a sequence in 
$[\varepsilon,(1-\varepsilon)/\beta]$, and set
\begin{equation}
\label{e:guad2012-10-28g}
\begin{array}{l}
\text{For}\;n=0,1,\ldots\\
\left\lfloor
\begin{array}{l}
s_{1,1,n}=x_{n}-\gamma_n\big(Cx_n+
\sum_{k=1}^KL_k^*v_{k,n}+a_{1,1,n}\big)\\
p_{1,1,n}=J_{\gamma_n A}(s_{1,1,n}+\gamma_nz)+b_{1,1,n}\\
\text{If}\;K_1\neq 0,\;\text{for}\;k=1,\ldots,K_1\\
\left\lfloor
\begin{array}{l}
s_{1,k+1,n}=y_{k,n}+\gamma_nv_{k,n}\\
p_{1,k+1,n}=J_{\gamma_n S_k}s_{1,k+1,n}+b_{1,k+1,n}\\
s_{2,k,n}=v_{k,n}-\gamma_n\big(y_{k,n}-L_kx_n+a_{2,k,n}\big)\\
p_{2,k,n}=s_{2,k,n}-\gamma_n\big(r_k+J_{\gamma_n^{-1}B_k}
(\gamma_n^{-1}s_{2,k,n}-r_k)+b_{2,k,n}\big)\\
q_{2,k,n}=p_{2,k,n}-\gamma_n\big(p_{1,k+1,n}-L_kp_{1,1,n}
+c_{2,k,n}\big)\\
v_{k,n+1}=v_{k,n}-s_{2,k,n}+q_{2,k,n}
\end{array}
\right.\\[1mm]
\text{If}\;K_1\neq K_2,\;\text{for}\;k=K_1+1,\ldots,K_2\\
\left\lfloor
\begin{array}{l}
s_{1,k+1,n}=y_{k,n}-\gamma_n\big(S_ky_{k,n}-v_{k,n}
+a_{1,k+1,n}\big)\\
p_{1,k+1,n}=s_{1,k+1,n}\\
s_{2,k,n}=v_{k,n}-\gamma_n\big(y_{k,n}-L_kx_n
+a_{2,k,n}\big)\\
p_{2,k,n}=s_{2,k,n}-\gamma_n\big(r_k+J_{\gamma_n^{-1}B_k}
(\gamma_n^{-1}s_{2,k,n}-r_k)+b_{2,k,n}\big)\\
q_{2,k,n}=p_{2,k,n}-\gamma_n\big(p_{1,k+1,n}-L_kp_{1,1,n}
+c_{2,k,n}\big)\\
v_{k,n+1}=v_{k,n}-s_{2,k,n}+q_{2,k,n}
\end{array}
\right.\\[1mm]
\text{If}\;K_2\neq K,\;\text{for}\;k=K_2+1,\ldots,K\\
\left\lfloor
\begin{array}{l}
s_{2,k,n}=v_{k,n}-\gamma_n\big(S_k^{-1}v_{k,n}-L_kx_n
+a_{2,k,n}\big)\\
p_{2,k,n}=s_{2,k,n}-\gamma_n\big(r_k+J_{\gamma_n^{-1}B_k}
(\gamma_n^{-1}s_{2,k,n}-r_k)+b_{2,k,n}\big)\\
q_{2,k,n}=p_{2,k,n}-\gamma_n\big(S_k^{-1}p_{2,k,n}-L_kp_{1,1,n}
+c_{2,k,n}\big)\\
v_{k,n+1}=v_{k,n}-s_{2,k,n}+q_{2,k,n}
\end{array}
\right.\\[1mm]
q_{1,1,n}=p_{1,1,n}-\gamma_n\big(Cp_{1,1,n}+
\sum_{k=1}^KL_k^*p_{2,k,n}+c_{1,1,n}\big)\\
x_{n+1}=x_{n}-s_{1,1,n}+q_{1,1,n}\\
\text{If}\;K_1\neq 0,\;\text{for}\;k=1,\ldots,K_1\\
\left\lfloor
\begin{array}{l}
q_{1,k+1,n}=p_{1,k+1,n}+\gamma_np_{2,k,n}\\
y_{k,n+1}=y_{k,n}-s_{1,k+1,n}+q_{1,k+1,n}
\end{array}
\right.\\
\text{If}\;K_1\neq K_2,\;\text{for}\;k=K_1+1,\ldots,K_2\\
\left\lfloor
\begin{array}{l}
q_{1,k+1,n}=p_{1,k+1,n}-\gamma_n\big(S_kp_{1,k+1,n}-p_{2,k,n}+
c_{1,k+1,n}\big)\\
y_{k,n+1}=y_{k,n}-s_{1,k+1,n}+q_{1,k+1,n}.
\end{array}
\right.\\
\end{array}
\right.\\
\end{array}
\end{equation}
Then the following hold for some solution $\overline{x}$ to 
\eqref{e:uH98j3h-07p} and some solution 
$(\overline{v_1},\ldots,\overline{v_K})$ to \eqref{e:uH98j3h-07d}.
\begin{enumerate}
\item
\label{p:1i}
$x_n\weakly\overline{x}~$ and 
$~(\forall k\in\{1,\ldots,K\})$ $v_{k,n}\weakly\overline{v_k}$.
\item 
\label{p:1iii}
Suppose that $A$ or $C$ is uniformly monotone at 
$\overline{x}$. Then $x_{n}\to\overline{x}$.
\item 
\label{p:1iv}
Suppose that, for some $l\in\{1,\ldots,K\}$, $B_l$ 
is couniformly monotone at $\overline{v_l}$. Then
$v_{l,n}\to\overline{v_l}$.
\item 
\label{p:1v}
Suppose that $K_2\neq K$ and that, for some 
$l\in\{K_2+1,\ldots,K\}$, $S_l$ is couniformly monotone at 
$\overline{v_l}$. Then $v_{l,n}\to\overline{v_l}$.
\end{enumerate}
\end{proposition}
\begin{proof}
We assume that $K_2\neq 0$ and consider the auxiliary problem 
\begin{multline}
\label{e:uH98j3h-07a}
\text{find}\;\;\overline{x}\in\HH,\:\overline{y_1}\in\GG_1,\:
\ldots,\:\overline{y_{K_2}}\in\GG_{K_2}
\;\:\text{such that}\\
\begin{cases}
z&\!\!\!\in
A\overline{x}+\Sum_{k=1}^{K_2}L_k^*\big(B_k(L_k\overline{x}
-\overline{y_k}-r_k)\big)+\Sum_{k=K_2+1}^KL_k^*\big(
(B_k\infconv S_k)(L_k\overline{x}-r_k)\big)+C\overline{x}\\
0&\!\!\!\in S_1\overline{y_1}-B_1(L_1\overline{x}-
\overline{y_1}-r_1)\\
&\!\!\vdots\\
0&\!\!\!\in S_{K_2}\overline{y_{K_2}}-B_{K_2}(L_{K_2}\overline{x}
-\overline{y_{K_2}}-r_{K_2})
\end{cases}
\end{multline}
together with the dual problem \eqref{e:uH98j3h-07d} (if $K_2=0$,
\eqref{e:uH98j3h-07a} should be replaced by \eqref{e:uH98j3h-07p} 
and the resulting simplifications in the proof are straightforward).
Let us show that solving the primal-dual problem 
\eqref{e:uH98j3h-07a}/\eqref{e:uH98j3h-07d} is a special case 
of Problem~\ref{prob:1} with 
\begin{equation}
\label{e:guadeloupe2012-10-29a}
\begin{cases}
m=K_2+1\\
\HH_1=\HH\\
A_1=A\\
C_1=C\\
\mu_1=\mu\\
\overline{x_1}=\overline{x}\\
z_1=z,
\end{cases}
(\forall k\in\{1,\ldots,K_2\})\quad
\begin{cases}
\HH_{k+1}=\GG_{k}\\
A_{k+1}=
\begin{cases}
S_{k},&\text{if}\;\;1\leq k\leq K_1;\\
0,&\text{if}\;\;K_1+1\leq k\leq K_2
\end{cases}\\[5mm]
C_{k+1}=
\begin{cases}
0,&\text{if}\;\;1\leq k\leq K_1;\\
S_{k},&\text{if}\;\;K_1+1\leq k\leq K_2
\end{cases}\\[5mm]
\mu_{k+1}=
\begin{cases}
0,&\text{if}\;\;1\leq k\leq K_1;\\
\beta_{k},&\text{if}\;\;K_1+1\leq k\leq K_2
\end{cases}\\
\overline{x_{k+1}}=\overline{y_{k}}\\
z_{k+1}=0,
\end{cases}
\end{equation}
and
\begin{equation}
\label{e:guadeloupe2012-10-29e}
(\forall k\in\{1,\ldots,K\})\quad
\begin{cases}
D_k=
\begin{cases}
\{0\}^{-1},&\text{if}\;\;1\leq k\leq K_2;\\
S_k,&\text{if}\;\;K_2+1\leq k\leq K\\
\end{cases}\\[5mm]
\nu_{k+1}=
\begin{cases}
0,&\text{if}\;\;1\leq k\leq K_2;\\
\beta_{k},&\text{if}\;\;K_2+1\leq k\leq K
\end{cases}\\
L_{k1}=L_k\\
(\forall i\in\{2,\ldots,K_2+1\})\:\; L_{ki}=
\begin{cases}
-\Id,&\text{if}\;\;i=k+1;\\
0,&\text{otherwise.}
\end{cases}
\end{cases}
\end{equation}
First, we note that, in this setting,  \eqref{e:IJ834hj8fr-24p} 
reduces to \eqref{e:uH98j3h-07a}, and \eqref{e:IJ834hj8fr-24d} to 
\eqref{e:uH98j3h-07d}.
Now define $\HHH$ and $\GGG$ as in 
\eqref{e:IJ834hj8fr-30}, let $x\in\HH$, 
let $(y_k)_{1\leq k\leq K_2}\in\bigoplus_{k=1}^{K_2}\GG_k$, set
$(x_i)_{1\leq i\leq m}=(x,y_1,\ldots,y_{K_2})\in\HHH$, set
$\boldsymbol{y}=(y_1,\ldots,y_{K_2},0,\ldots,0)\in\GGG$,
and set $\lambda=1+\sum_{k=1}^{K_2}\|L_k\|^2$.
Then, using the Cauchy-Schwarz inequality in $\RR^2$,
\begin{multline}
\label{e:roma1}
\sum_{k=1}^K\bigg\|\sum_{i=1}^mL_{ki}x_i\bigg\|^2
=\|(L_kx)_{1\leq k\leq K_2}-\boldsymbol{y}\|^2
\leq\big(\|\boldsymbol{y}\|+\|(L_kx)_{1\leq k\leq K_2}\|\big)^2\\
\leq\left(\|\boldsymbol{y}\|+\sqrt{\sum_{k=1}^{K_2}
\|L_k\|^2}\,\|x\|\right)^2
\leq\bigg(1+\sum_{k=1}^{K_2}\|L_k\|^2\bigg)
\big(\|\boldsymbol{y}\|^2+\|x\|^2\big)=
\lambda\sum_{i=1}^m\|x_i\|^2.
\end{multline}
Thus \eqref{e:5h8Njiq-11a} is a special case of
specializes to \eqref{e:uH98j3h-07c}. On the other hand,
by assumption, 
\eqref{e:uH98j3h-07p} has a solution, say $x$. Therefore, there 
exist $v_1\in\GG_1$, \ldots, $v_{K_2}\in\GG_{K_2}$ such that
\begin{equation}
\label{e:uH98j3h-12a}
\begin{cases}
z\in Ax+\Sum_{k=1}^{K_2}L_k^*v_k+\Sum_{k=K_2+1}^KL_k^*
\big((B_k\infconv S_k)(L_k{x}-r_k)\big)+Cx\\
(\forall k\in\{1,\ldots,K_2\})\quad
v_k\in(B_k\infconv S_k)(L_k{x}-r_k).
\end{cases}
\end{equation}
Therefore, in view of \eqref{e:parasum}, there exist 
$y_1\in\GG_1$, \ldots, $y_{K_2}\in\GG_{K_2}$ such that
\begin{equation}
\label{e:uH98j3h-12A}
\begin{cases}
z\in Ax+\Sum_{k=1}^{K_2}L_k^*v_k+\Sum_{k=K_2+1}^KL_k^*
\big((B_k\infconv S_k)(L_k{x}-r_k)\big)+Cx\\
(\forall k\in\{1,\ldots,K_2\})\quad y_k\in
S_k^{-1}v_k\quad\text{and}\quad L_k{x}-y_k-r_k\in B_k^{-1}v_k,
\end{cases}
\end{equation}
which implies that
\begin{equation}
\label{e:uH98j3h-14a}
\begin{cases}
z\in Ax+\Sum_{k=1}^{K_2}L_k^*v_k+
\Sum_{k=K_2+1}^KL_k^*\big((B_k\infconv S_k)(L_k{x}-r_k)\big)+Cx\\
(\forall k\in\{1,\ldots,K_2\})\quad v_k\in
S_ky_k\quad\text{and}\quad v_k\in B_k(L_k{x}-y_k-r_k),
\end{cases}
\end{equation}
and therefore that
\begin{equation}
\label{e:uH98j3h-14b}
\begin{cases}
z\in Ax+\Sum_{k=1}^{K_2}L_k^*\big(B_k(L_k{x}-y_k-r_k)\big)+
\Sum_{k=K_2+1}^KL_k^*\big((B_k\infconv S_k)(L_k{x}-r_k)\big)+Cx\\
(\forall k\in\{1,\ldots,K_2\})\quad 0\in
S_ky_k -B_k(L_k{x}-y_k-r_k).
\end{cases}
\end{equation}
This shows that \eqref{e:uH98j3h-07a} possesses a solution.
Next, upon defining 
\begin{equation}
\label{e:5h8Njiq-18x}
(\forall n\in\NN)\quad x_{1,n}=x_n\quad\text{and}\quad
(\forall k\in\{1,\ldots,K_2\})\quad 
\begin{cases}
x_{k+1,n}=y_{k,n};&\\
a_{1,k+1,n}=0,&\text{if}\;\;1\leq k\leq K_1;\\
b_{1,k+1,n}=0,&\text{if}\;\;K_1+1\leq k\leq K_2;\\
c_{1,k+1,n}=0,&\text{if}\;\;1\leq k\leq K_1,
\end{cases}
\end{equation}
we see that \eqref{e:5h8Njiq-11b} specializes to 
\eqref{e:guad2012-10-28g}. Hence, in view of 
\eqref{e:guadeloupe2012-10-29a}--\eqref{e:guadeloupe2012-10-29e} 
and Theorem~\ref{t:1}\ref{t:1iia}--\ref{t:1iid}, there exist a 
solution $(\overline{x},\overline{y_1},\ldots,\overline{y_{K_2}})$ 
to \eqref{e:uH98j3h-07a} and a solution 
$(\overline{v_1},\ldots,\overline{v_K})$ to \eqref{e:uH98j3h-07d} 
such that
\begin{equation}
\label{e:5h8Njiq-18a}
x_n\weakly\overline{x}\quad\text{and}\quad
(\forall k\in\{1,\ldots,K\})\quad v_{k,n}\weakly\overline{v_k},
\end{equation}
with
\begin{multline}
\label{e:uH98j3h-07m}
z-\sum_{k=1}^KL_{k}^*\overline{v_k}\in A\overline{x}+C\overline{x},
\quad (\forall k\in\{1,\ldots,K_2\})\quad
\begin{cases}
L_k\overline{x}-\overline{y_k}-r_k\in B_k^{-1}\overline{v_k}\\
\overline{v_k}\in S_k\overline{y_k},
\end{cases}\\
\quad\text{and}\quad
(\forall k\in\{K_2+1,\ldots,K\})\quad
L_k\overline{x}-r_k\in B_k^{-1}\overline{v_k}
+S_k^{-1}\overline{v_k}.
\end{multline}
Since the strong convergence claims \ref{p:1iii}--\ref{p:1v} 
are immediate consequences of 
Theorem~\ref{t:1}\ref{t:1iie}--\ref{t:1iif}, it remains to show 
that $\overline{x}$ solves \eqref{e:uH98j3h-07p}.
We derive from \eqref{e:uH98j3h-07m} that,
for every $k\in\{1,\ldots,K_2\}$, 
$L_k\overline{x}-\overline{y_k}-r_k\in B_k^{-1}\overline{v_k}$
and $\overline{y_k}\in S_k^{-1}\overline{v_k}$, and,
for every $k\in\{K_2+1,\ldots,K\}$, 
$L_k\overline{x}-r_k\in B_k^{-1}\overline{v_k}+
S_k^{-1}\overline{v_k}$. Altogether,
\begin{equation}
\label{e:5h8Njiq-18z}
(\forall k\in\{1,\ldots,K\})\quad 
L_k\overline{x}-r_k\in\big(B_k^{-1}+S_k^{-1}\big)\overline{v_k}
\end{equation}
and, therefore,
\begin{equation}
\label{e:5h8Njiq-18c}
\sum_{k=1}^KL_k^*\overline{v_k}\in\sum_{k=1}^KL_k^*
\big(\big(B_k^{-1}+S_k^{-1}\big)^{-1}
(L_k\overline{x}-r_k)\big)=\sum_{k=1}^KL_k^*
\big((B_k\infconv S_k)(L_k\overline{x}-r_k)\big).
\end{equation}
Thus, since \eqref{e:uH98j3h-07m} also asserts that 
$z-\sum_{k=1}^KL_{k}^*\overline{v_k}\in A\overline{x}
+C\overline{x}$, we conclude that $\overline{x}$ solves 
\eqref{e:uH98j3h-07p}. 
\end{proof}

\begin{remark}
\label{r:guadeloupe2012-10-30}
Problem~\ref{prob:2} encompasses more general
scenarios than that of \cite{Svva12}, which corresponds to the
case when $K_1=K_2=0$, i.e., when all the operators 
$(D_k^{-1})_{1\leq k\leq K}$ are restricted to be Lipschitzian.
This extension has been made possible by reformulating the 
original primal problem \eqref{e:uH98j3h-07p}, which involves only 
one variable, as the extended primal problem \eqref{e:uH98j3h-07a}, 
in which we added $K_2$ auxiliary variables. We also note that
Algorithm~\eqref{e:guad2012-10-28g} uses all the single-valued 
operators present in Problem~\ref{prob:2}, including 
$(S_k)_{K_1+1\leq k\leq K_2}$ and $(S_k^{-1})_{K_2+1\leq k\leq K}$, 
through explicit steps. 
\end{remark}

\section{Relaxation of inconsistent common zero problems}
\label{sec:5}

A common problem in nonlinear analysis is to find a common 
zero of maximally monotone operators $A$ and 
$(B_k)_{1\leq k\leq K}$ acting on a real Hilbert space $\HH$
\cite{Cras95,Dye92a,Lehd99}, i.e.,
\begin{equation}
\label{e:cras1995}
\text{find}\;\;\overline{x}\in\HH\;\;\text{such that}\;\;
0\in A\overline{x}\cap\bigcap_{k=1}^KB_k\overline{x}.
\end{equation}
In many situations, this problem may be inconsistent
(see \cite{Sign99} and the references therein) and it must be
approximated. We study the following relaxation of 
\eqref{e:cras1995}, together with its dual problem.

\begin{problem}
\label{prob:3}
Let $\HH$ be a real Hilbert space, let $K$ be a strictly positive
integer, let $A\colon\HH\to 2^{\HH}$ be maximally monotone, and, 
for every $k\in\{1,\ldots,K\}$, let $S_k\colon\HH\to 2^{\HH}$ be a
maximally monotone operator such that $S_k^{-1}$ is at most
single-valued and strictly monotone, with $S_k^{-1} 0=\{0\}$.
It is assumed that the inclusion
\begin{equation}
\label{e:uH98j3h-29p}
\text{find}\;\;\overline{x}\in\HH\;\;\text{such that}\;\;
0\in A\overline{x}+\sum_{k=1}^K(B_k\infconv S_k)\overline{x}
\end{equation}
possesses at least one solution. Solve \eqref{e:uH98j3h-29p} 
together with the dual problem
\begin{multline}
\label{e:uH98j3h-29d}
\text{find}\;\;\overline{u_1}\in\HH,\:\ldots,\:
\overline{u_K}\in\HH\;\:\text{such that}\\
(\forall k\in\{1,\ldots,K\})\quad
0\in -A^{-1}\bigg(-\Sum_{l=1}^K
\overline{u_l}\bigg)+B_k^{-1}\overline{u_k}+
S_k^{-1}\overline{u_k}.
\end{multline}
\end{problem}

First, we justify the fact that \eqref{e:uH98j3h-29p} is indeed a 
relaxation of \eqref{e:cras1995}.

\begin{proposition}
\label{p:uH98j3h-30}
In the setting of Problem~\ref{prob:3}, set
$Z=(\zer A)\cap\bigcap_{k=1}^K\zer B_k$ and  
suppose that $Z\neq\emp$.
Then the set of solutions to the primal problem
\eqref{e:uH98j3h-29p} is $Z$.
\end{proposition}
\begin{proof}
It is clear that every point in $Z$ solves \eqref{e:uH98j3h-29p}.
Conversely, let $x$ be a solution to \eqref{e:uH98j3h-29p} and let 
$z\in Z$. We first note that the operators 
$(B_k\infconv S_k)_{1\leq k\leq K}$ are at most single-valued.
Indeed, let $k\in\{1,\ldots,K\}$ and let $(y,p)$ and $(y,q)$ be
in $\gra(B_k\infconv S_k)$. Then we must show that $p=q$.
We have  $p=(B_k\infconv S_k)y$ $\Leftrightarrow$ 
$y\in B_k^{-1}p+S_k^{-1}p$ $\Leftrightarrow$ 
$y-S_k^{-1}p\in B_k^{-1}p$. Likewise, $y-S_k^{-1}q\in B_k^{-1}q$ 
and, by monotonicity of $B_k$,
$-\scal{p-q}{S_k^{-1}p-S_k^{-1}q}=
\scal{p-q}{(y-S_k^{-1}p)-(y-S_k^{-1}q)}\geq 0$. Consequently,
by strict monotonicity of $S_k^{-1}$, 
$\scal{p-q}{S_k^{-1}p-S_k^{-1}q}=0$ and $p=q$.
Hence, since $x$ solves \eqref{e:uH98j3h-29p}, there exists
$(p_k)_{0\leq k\leq K}\in\HH^{K+1}$ such that
\begin{equation}
\label{e:uH98j3h-30a}
\sum_{k=0}^K p_k=0,\quad p_0\in Ax,\quad\text{and}\quad
(\forall k\in\{1,\ldots,K\})\quad
p_k=(B_k\infconv S_k)x.
\end{equation}
Therefore, we have
\begin{equation}
\label{e:uH98j3h-30b}
p_0\in Ax,\;0\in Az,
\quad\text{and}\quad 
(\forall k\in\{1,\ldots,K\})\quad
p_k\in B_k\big(x-S^{-1}_kp_k\big) 
\quad\text{and}\quad 0\in B_k z,
\end{equation}
and, by monotonicity of the operators $A$ and
$(B_k)_{1\leq k\leq K}$,
\begin{equation}
\label{e:uH98j3h-30c}
\scal{x-z}{p_0}\geq 0
\quad\text{and}\quad 
(\forall k\in\{1,\ldots,K\})\quad
\scal{x-S^{-1}_kp_k-z}{p_k}\geq 0.
\end{equation}
Hence, since $\sum_{k=0}^K p_k=0$, it follows from the monotonicity
of the operators $(S_k^{-1})_{1\leq k\leq K}$ that
\begin{align}
\label{e:uH98j3h-30d}
0&\geq -\sum_{k=1}^K\scal{p_k-0}{S^{-1}_kp_k-S^{-1}_k0}
\nonumber\\
&=\sum_{k=0}^K\scal{x-z}{p_k}-\sum_{k=1}^K\scal{S^{-1}_kp_k}{p_k}
\nonumber\\
&=\scal{x-z}{p_0}+\sum_{k=1}^K\scal{x-S^{-1}_kp_k-z}{p_k}
\nonumber\\
&\geq 0.
\end{align}
Thus, $\sum_{k=1}^K\scal{p_k-0}{S^{-1}_kp_k-S^{-1}_k0}=0$ and,
therefore, 
\begin{equation}
(\forall k\in\{1,\ldots,K\})\quad 
\scal{p_k-0}{S^{-1}_kp_k-S^{-1}_k0}=0. 
\end{equation}
The strict monotonicity 
of the operators $(S_k^{-1})_{1\leq k\leq K}$ implies that for 
every $k\in\{1,\ldots,K\}$ $p_k=0$, i.e., 
$x\in B_k^{-1}p_k+S_k^{-1}p_k=B_k^{-1}0+S_k^{-1}0=B_k^{-1}0$. 
In turn, $p_0=-\sum_{k=1}^Kp_k=0$, i.e., $x\in A^{-1}0$.
Altogether, $x\in Z$.
\end{proof}

\begin{remark}
\label{r:nov2012}
Suppose that in Problem~\ref{prob:3} we set, for every 
$k\in\{1,\ldots,K\}$, $S_k=\gamma_k^{-1}\Id$ where 
$\gamma_k\in\RPP$, i.e., $B_k\infconv S_k=\moyo{B_k}{\gamma_k}$ 
is the Yosida approximation of $B_k$ of index $\gamma_k$
\cite[Proposition~23.6(ii)]{Livre1}. 
Then \eqref{e:uH98j3h-29p} reduces to the setting investigated 
in \cite[Section~6.3]{Opti04}, namely
\begin{equation}
\label{e:uH98j3h-30p}
\text{find}\;\;\overline{x}\in\HH\;\;\text{such that}\;\;
0\in A\overline{x}+\sum_{k=1}^K\moyo{B_k}{\gamma_k}\overline{x},
\end{equation}
which itself covers the frameworks of 
\cite{Reic05,Sign99,Mahe93,Merc80} and the references therein.
In this case, Proposition~\ref{p:uH98j3h-30} specializes to
\cite[Proposition~6.10]{Opti04}. Now let us further specialize to
the case when $\HH=\RR^N$, $A=0$, and
\begin{equation}
(\forall k\in\{1,\ldots,K\})\quad 
\begin{cases}
\gamma_k=1\\
B_k\colon x\mapsto
\begin{cases}
\spa\{u_k\},&\text{if}\;\;\scal{x}{u_k}=\rho_k;\\
\emp,&\text{if}\;\;\scal{x}{u_k}\neq\rho_k,
\end{cases}
\end{cases}
\text{where}\quad 
\begin{cases}
u_k\in\RR^N\\
\|u_k\|=1\\
\rho_k\in\RR.
\end{cases}
\end{equation}
Then \eqref{e:cras1995} amounts to solving the system of linear 
equalities
\begin{equation}
\label{e:cauchy}
\text{find}\;\;\overline{x}\in\RR^N\;\;\text{such that}\;\;
(\forall k\in\{1,\ldots,K\})\quad\scal{\overline{x}}{u_k}=\rho_k,
\end{equation}
whereas \eqref{e:uH98j3h-29p} amounts to solving the least-squares
problem
\begin{equation}
\label{e:legendre}
\minimize{x\in\RR^N}{\sum_{k=1}^m|\scal{x}{u_k}-\rho_k|^2}.
\end{equation}
The idea of relaxing \eqref{e:cauchy} to \eqref{e:legendre} is 
due to Legendre \cite{Lege05} and Gauss \cite{Gaus09}.
\end{remark}

To solve Problem~\ref{prob:3}, we use Proposition~\ref{p:1}
to derive the following algorithm.

\begin{proposition}
\label{p:2}
Consider the setting of Problem~\ref{prob:3}. Let 
$(b_{1,1,n})_{n\in\NN}$ and, for every $k\in\{1,\ldots,K\}$, 
$(b_{1,k+1,n})_{n\in\NN}$ and $(b_{2,k,n})_{n\in\NN}$ be absolutely
summable sequences in $\HH$. Let $x_{0}\in\HH$, 
$(y_{k,0})_{1\leq k\leq K}\in\HH^K$, 
$(v_{k,0})_{1\leq k\leq K}\in\HH^K$, 
and $\varepsilon\in\,]0,1/(\sqrt{K+1}+1)[$\,, 
let $(\gamma_n)_{n\in\NN}$ be a sequence in 
$[\varepsilon,(1-\varepsilon)/\sqrt{K+1}]$, and set
\begin{equation}
\label{e:kl74Gh9c6-01g}
\begin{array}{l}
\text{For}\;n=0,1,\ldots\\
\left\lfloor
\begin{array}{l}
p_{1,1,n}=J_{\gamma_n A}\big(x_{n}-\gamma_n
\sum_{k=1}^Kv_{k,n}\big)+b_{1,1,n}\\
\text{For}\;k=1,\ldots,K\\
\left\lfloor
\begin{array}{l}
p_{1,k+1,n}=J_{\gamma_n S_k}(y_{k,n}+\gamma_nv_{k,n})+b_{1,k+1,n}\\
s_{2,k,n}=v_{k,n}-\gamma_n(y_{k,n}-x_n)\\
p_{2,k,n}=s_{2,k,n}-\gamma_n\big(J_{\gamma_n^{-1}B_k}
(\gamma_n^{-1}s_{2,k,n})+b_{2,k,n}\big)\\
v_{k,n+1}=v_{k,n}-s_{2,k,n}+p_{2,k,n}-\gamma_n\big(p_{1,k+1,n}-
p_{1,1,n}\big)
\end{array}
\right.\\[1mm]
x_{n+1}=p_{1,1,n}+\gamma_n\sum_{k=1}^K(v_{k,n}-p_{2,k,n})\\
\text{For}\;k=1,\ldots,K\\
\left\lfloor
\begin{array}{l}
y_{k,n+1}=p_{1,k+1,n}+\gamma_n(p_{2,k,n}-v_{k,n})
\end{array}
\right.\\
\end{array}
\right.\\
\end{array}
\end{equation}
Then the following hold for some solution $\overline{x}$ to 
\eqref{e:uH98j3h-29p} and some solution 
$(\overline{v_1},\ldots,\overline{v_K})$ to \eqref{e:uH98j3h-29d}.
\begin{enumerate}
\item
\label{p:2i}
$x_n\weakly\overline{x}$ and 
$(\forall k\in\{1,\ldots,K\})$ $v_{k,n}\weakly\overline{v_k}$.
\item
\label{p:2i'}
Suppose that $A$ is uniformly monotone at 
$\overline{x}$. Then $x_{n}\to\overline{x}$.
\item
\label{p:2ii}
Suppose that, for some $l\in\{1,\ldots,K\}$, $B_l$ is couniformly 
monotone at $\overline{v_l}$. Then $v_{l,n}\to\overline{v_l}$.
\end{enumerate}
\end{proposition}
\begin{proof}
Problem~\ref{prob:3} is a special case of Problem~\ref{prob:2}
with $K_1=K_2=K$, $z=0$, $C=0$, $\mu=0$, $\beta=\sqrt{K+1}$, and 
$(\forall k\in\{1,\ldots,K\})$ $\GG_k=\HH$, $L_k=\Id$, and $r_k=0$.
In this context, \eqref{e:guad2012-10-28g} can be reduced to 
\eqref{e:kl74Gh9c6-01g}, and the claims therefore follow from
Proposition~\ref{p:1}.
\end{proof}

\begin{remark}
\label{r:kl74Gh9c6-02}
For brevity, we have presented an algorithm for solving
Problem~\ref{prob:3} in its general form. However,
if some of the operators $(S_k)_{1\leq k\leq K}$ or their inverses
are Lipschitzian, we can apply Proposition~\ref{p:1} with 
$K_1\neq K$ and/or $K_2\neq K$ to obtain a more efficient 
algorithm in which each Lipschitzian operator is used through an 
explicit step, rather than through its resolvent.
\end{remark}

\section{Multivariate structured convex minimization problems}
\label{sec:6}

We derive from Theorem~\ref{t:1} a primal-dual minimization
algorithm for multivariate convex minimization problems involving 
infimal convolutions and composite functions.

\begin{problem}
\label{prob:5}
Let $m$ and $K$ be strictly positive integers, let
$(\HH_i)_{1\leq i\leq m}$ and $(\GG_k)_{1\leq k\leq K}$ be real 
Hilbert spaces, let $(\mu_i)_{1\leq i\leq m}\in\RP^m$, and let 
$(\nu_k)_{1\leq i\leq K}\in\RPP^K$. For every $i\in\{1,\ldots,m\}$ 
and $k\in\{1,\ldots,K\}$, let $h_i\colon\HH_i\to\RR$ be convex and 
differentiable and such that $\nabla h_i$ is $\mu_i$-Lipschitzian, 
let $f_i\in\Gamma_0(\HH_i)$, let $g_k\in\Gamma_0(\GG_k)$, 
let $\ell_k\in\Gamma_0(\GG_k)$ be $1/\nu_k$-strongly convex, 
let $z_i\in\HH_i$, let $r_k\in\GG_k$, and let 
$L_{ki}\in\BL(\HH_i,\GG_k)$. Set
$\beta=\text{max}\Big\{\underset{1\leq i\leq m}{\text{max}}\mu_i,
\underset{1\leq k\leq K}{\text{max}}\nu_k\Big\}+
\sqrt{\lambda}>0$, where
$\lambda\in\left[\text{sup}_{\sum_{i=1}^m\|x_i\|^2\leq 1}
\sum_{k=1}^K\|\sum_{i=1}^mL_{ki}x_i\|^2,\pinf\right[$,
and assume that 
\begin{equation}
\label{e:5h8Njiq-21a}
(\forall i\in\{1,\ldots,m\})\quad
z_i\in\ran\bigg(\partial f_i+\sum_{k=1}^KL_{ki}^*\circ
(\partial g_k\infconv\partial\ell_k)
\circ\bigg(\sum_{j=1}^mL_{kj}\cdot-r_k\bigg)+\nabla h_i\bigg).
\end{equation}
Solve the primal problem
\begin{equation}
\label{e:5h8Njiq-23p}
\minimize{x_1\in\HH_1,\ldots,\,x_m\in\HH_m}{\sum_{i=1}^mf_i(x_i)
+\sum_{k=1}^K(g_k\infconv\ell_k)\bigg(\sum_{i=1}^m
L_{ki}x_i-r_k\bigg)+
\sum_{i=1}^m\big(h_i(x_i)-\scal{x_i}{z_i}\big)},
\end{equation}
together with the dual problem
\begin{equation}
\label{e:5h8Njiq-23d}
\minimize{v_1\in\GG_1,\ldots,\,v_K\in\GG_K}{\sum_{i=1}^m
\big(f_i^*\infconv h_i^*)\bigg(z_i-\sum_{k=1}^KL_{ki}^*v_k\bigg)
+\sum_{k=1}^K\big(g^*_k(v_k)+\ell^*_k(v_k)+\scal{v_k}{r_k}\big)}.
\end{equation}
\end{problem}

\begin{remark}
\label{r:3jqocv9k-12-09}
Problem~\ref{prob:5} extends significantly the multivariate 
minimization framework of \cite{Sico10,Nmtm09}. There, 
$(h_i)_{1\leq i\leq m}$ were the zero function,
$(\ell_k)_{1\leq k\leq K}$ were the function $\iota_{\{0\}}$,
and $(g_k)_{1\leq k\leq K}$ were differentiable 
everywhere with a Lipschitzian gradient. 
Finally, no dual problem was considered.
\end{remark}

\begin{proposition}
\label{p:3jqocv9k-12-10}
Consider the setting of Problem~\ref{prob:5}.
Suppose that \eqref{e:5h8Njiq-23p} has a solution, and set 
\begin{equation}
\label{e:cq}
E=\Menge{\bigg(\sum_{i=1}^mL_{ki}x_i-y_k\bigg)_{1\leq k\leq K}}
{
\begin{cases}
(\forall i\in\{1,\ldots,m\})\:\:x_i\in\dom f_i\\
(\forall k\in\{1,\ldots,K\})\:\:y_k\in
\dom g_k+\dom\ell_k
\end{cases}
}.
\end{equation}
Then \eqref{e:5h8Njiq-21a} is satisfied in each of the following
cases.
\begin{enumerate}
\item
\label{p:3jqocv9k-12-10i}
$(r_k)_{1\leq k\leq K}\in\sri E$.
\item
\label{p:3jqocv9k-12-10ii--}
$E-(r_k)_{1\leq k\leq K}$ is a closed vector subspace.
\item
\label{p:3jqocv9k-12-10ii-}
For every $i\in\{1,\ldots,m\}$, $f_i$ is real-valued and,
for every $k\in\{1,\ldots,K\}$, the operator
$\bigoplus_{j=1}^m\HH_j\to\GG_k\colon(x_j)_{1\leq j\leq m}\mapsto
\sum_{j=1}^mL_{kj}x_j$ is surjective.
\item
\label{p:3jqocv9k-12-10ii}
For every $k\in\{1,\ldots,K\}$, $g_k$ or $\ell_k$ is real-valued.
\item
\label{p:3jqocv9k-12-10iii}
$(\HH_i)_{1\leq i\leq m}$ and $(\GG_k)_{1\leq k\leq K}$ are 
finite-dimensional, and $(\forall i\in\{1,\dots,m\})
(\exi x_i\in\reli\dom f_i)(\forall k\in\{1,\ldots,K\})$
$\sum_{i=1}^mL_{ki}x_i-r_k\in\reli\dom g_k+\reli\dom\ell_k$.
\end{enumerate}
\end{proposition}
\begin{proof}
Define $\HHH$ and $\GGG$ as in \eqref{e:IJ834hj8fr-30},
and $\boldsymbol{L}$, $\boldsymbol{z}$, and $\boldsymbol{r}$ as 
in \eqref{e:5h8Njiq-12b}. Set
\begin{equation}
\label{e:uH98j3h-20d}
\begin{cases}
\boldsymbol{f}\colon\HHH\to\RX\colon\boldsymbol{x}\mapsto
\sum_{i=1}^mf_i(x_i)\quad\text{and}\quad
\boldsymbol{h}\colon\HHH\to\RR\colon\boldsymbol{x}\mapsto
\sum_{i=1}^mh_i(x_i),\\
\boldsymbol{g}\colon\GGG\to\RX\colon\boldsymbol{y}\mapsto
\sum_{k=1}^Kg_k(y_k)\quad\text{and}\quad
\boldsymbol{\ell}\colon\GGG\to\RX\colon\boldsymbol{y}\mapsto
\sum_{k=1}^K\ell_k(y_k).
\end{cases}
\end{equation}
Then \eqref{e:cq} and \cite[Proposition~12.6(ii)]{Livre1} yield
\begin{align}
E
&=\menge{\boldsymbol{L}\boldsymbol{x}-\boldsymbol{y}}
{\boldsymbol{x}\in\dom\boldsymbol{f}\;\text{and}\;\boldsymbol{y}
\in\dom\boldsymbol{g}+\dom\boldsymbol{\ell}}
\nonumber\\
&=\boldsymbol{L}\big(\dom\boldsymbol{f}\big)-
\big(\dom\boldsymbol{g}+\dom\boldsymbol{\ell}\big)
\label{e:erice}\\
&=\boldsymbol{L}\big(\dom(\boldsymbol{f}+\boldsymbol{h}-
\scal{\cdot}{\boldsymbol{z}})\big)-\dom\big(\boldsymbol{g}\infconv
\boldsymbol{\ell}\big).
\label{e:mam}
\end{align}

\ref{p:3jqocv9k-12-10i}:
Since the functions $(\ell_k)_{1\leq k\leq K}$ are strongly convex,
so is $\boldsymbol{\ell}$. Hence, $\dom\boldsymbol{\ell}^*=\GGG$ 
\cite[Propositions~11.16 and 14.15]{Livre1} and therefore 
\cite[Propositions~15.7(iv) and 24.27]{Livre1} imply that
$\partial \boldsymbol{g}\infconv\partial\boldsymbol{\ell}=
\partial(\boldsymbol{g}\infconv\boldsymbol{\ell})$ and 
$\boldsymbol{g}\infconv\boldsymbol{\ell}\in\Gamma_0(\GGG)$.
On the other hand, \eqref{e:mam} yields
$\boldsymbol{0}\in\sri(\boldsymbol{L}
(\dom(\boldsymbol{f}+\boldsymbol{h}-
\scal{\cdot}{\boldsymbol{z}}))-
\dom(\boldsymbol{g}\infconv\boldsymbol{\ell})
(\cdot-\boldsymbol{r}))$.
Thus, we derive from \cite[Theorem~16.37(i)]{Livre1} that
\begin{align}
\label{e:boardwalk}
\partial\boldsymbol{f}+\boldsymbol{L}^*\circ(\partial\boldsymbol{g}
\infconv\partial\boldsymbol{\ell}_k)\circ(\boldsymbol{L}\cdot-
\boldsymbol{r})+\nabla\boldsymbol{h}-\boldsymbol{z}
&=\partial\big(\boldsymbol{f}+\boldsymbol{h}-
\scal{\cdot}{\boldsymbol{z}}\big)+
\boldsymbol{L}^*\circ\partial(\boldsymbol{g}
\infconv\boldsymbol{\ell})\circ(\boldsymbol{L}\cdot-\boldsymbol{r})
\nonumber\\
&=\partial\big(\boldsymbol{f}+\boldsymbol{h}-
\scal{\cdot}{\boldsymbol{z}}+(\boldsymbol{g}
\infconv\boldsymbol{\ell})\circ(\boldsymbol{L}\cdot-
\boldsymbol{r})\big).
\end{align}
Since \eqref{e:5h8Njiq-23p} has a solution 
and is equivalent to minimizing $\boldsymbol{f}+\boldsymbol{h}-
\scal{\cdot}{\boldsymbol{z}}+(\boldsymbol{g}
\infconv\boldsymbol{\ell})\circ(\boldsymbol{L}\cdot-r)$ over 
$\HHH$, Fermat's rule \cite[Theorem~16.2]{Livre1} implies that 
$\boldsymbol{0}\in\ran\,\partial(\boldsymbol{f}+\boldsymbol{h}-
\scal{\cdot}{\boldsymbol{z}}+(\boldsymbol{g}\infconv
\boldsymbol{\ell})\circ(\boldsymbol{L}\cdot- \boldsymbol{r}))$. 
Hence \eqref{e:boardwalk} yields
$\boldsymbol{z}\in\ran(\partial\boldsymbol{f}+\boldsymbol{L}^*
\circ(\partial\boldsymbol{g}\infconv\partial\boldsymbol{\ell}_k)
\circ(\boldsymbol{L}\cdot-\boldsymbol{r})+\nabla\boldsymbol{h})$
and we conclude that \eqref{e:5h8Njiq-21a} is satisfied.

\ref{p:3jqocv9k-12-10ii--}$\Rightarrow$%
\ref{p:3jqocv9k-12-10i}: \cite[Proposition~6.19(i)]{Livre1}. 

\ref{p:3jqocv9k-12-10ii-}$\Rightarrow$%
\ref{p:3jqocv9k-12-10i}: We have 
$\boldsymbol{L}(\dom\boldsymbol{f})=\boldsymbol{L}(\HHH)=\GGG$. 
Hence, \eqref{e:erice} yields $E=\GGG$.

\ref{p:3jqocv9k-12-10ii}$\Rightarrow$%
\ref{p:3jqocv9k-12-10i}: We have
$\dom\boldsymbol{g}+\dom\boldsymbol{\ell}=\GGG$. Hence,
\eqref{e:erice} yields $E=\GGG$.

\ref{p:3jqocv9k-12-10iii}$\Rightarrow$%
\ref{p:3jqocv9k-12-10i}: Since $\text{dim}\,\GGG<\pinf$, 
$\sri E=\reli E$. On the other hand, by \eqref{e:erice} and 
\cite[Corollary~6.15]{Livre1},
\begin{equation}
\label{e:2011-07-08a}
\reli E=\reli\big(\boldsymbol{L}\big(\dom\boldsymbol{f}\big)-
\dom\boldsymbol{g}-\dom\boldsymbol{\ell}\big)
=\boldsymbol{L}\big(\reli\dom\boldsymbol{f}\big)-
\reli\dom\boldsymbol{g}-\reli\dom\boldsymbol{\ell}.
\end{equation}
Thus, $\boldsymbol{r}\in\sri E$ $\Leftrightarrow$
$(\exi\boldsymbol{x}\in\reli\dom\boldsymbol{f}=
\cart_{\!i=1}^{\!m}\reli\dom f_i)$
$\boldsymbol{L}\boldsymbol{x}-\boldsymbol{r}\in
\reli\dom\boldsymbol{g}+\reli\dom\boldsymbol{\ell}=
\cart_{\!k=1}^{\!K}(\reli\dom g_k+\reli\dom\ell_k)$.
\end{proof}

\begin{proposition}
\label{p:5}
Consider the setting of Problem~\ref{prob:5}.
For every $i\in\{1,\ldots,m\}$, let $(a_{1,i,n})_{n\in\NN}$,
$(b_{1,i,n})_{n\in\NN}$, and $(c_{1,i,n})_{n\in\NN}$ be 
absolutely summable sequences in $\HH_i$ and, for every 
$k\in\{1,\ldots,K\}$, let $(a_{2,k,n})_{n\in\NN}$, 
$(b_{2,k,n})_{n\in\NN}$, and $(c_{2,k,n})_{n\in\NN}$ be 
absolutely summable sequences in $\GG_k$. Furthermore, 
let $x_{1,0}\in\HH_1$, \ldots, $x_{m,0}\in\HH_m$, 
$v_{1,0}\in\GG_1$, \ldots, $v_{K,0}\in\GG_K$, let
$\varepsilon\in\left]0,1/(\beta+1)\right[$, 
let $(\gamma_n)_{n\in\NN}$ be a sequence in 
$[\varepsilon,(1-\varepsilon)/\beta]$, and set
\begin{equation}
\label{e:5h8Njiq-23b}
\begin{array}{l}
\text{For}\;n=0,1,\ldots\\
\left\lfloor
\begin{array}{l}
\text{For}\;i=1,\ldots,m\\
\left\lfloor
\begin{array}{l}
s_{1,i,n}=x_{i,n}-\gamma_n\big(\nabla h_i(x_{i,n})+
\sum_{k=1}^KL_{ki}^*v_{k,n}+a_{1,i,n}\big)\\[1mm]
p_{1,i,n}=\prox_{\gamma_n f_i}(s_{1,i,n}+\gamma_nz_i)
+b_{1,i,n}\\[1mm]
\end{array}
\right.\\[1mm]
\text{For}\;k=1,\ldots,K\\
\left\lfloor
\begin{array}{l}
s_{2,k,n}=v_{k,n}-\gamma_n\big(\nabla\ell_k^*(v_{k,n})-
\sum_{i=1}^mL_{ki}x_{i,n}+a_{2,k,n}\big)\\[2mm]
p_{2,k,n}=s_{2,k,n}-\gamma_n\big(r_k+\prox_{\gamma_n^{-1}g_k}
(\gamma_n^{-1}s_{2,k,n}-r_k)+b_{2,k,n}\big)\\[2mm]
q_{2,k,n}=p_{2,k,n}-\gamma_n\big(\nabla\ell_k^*(p_{2,k,n})-
\sum_{i=1}^mL_{ki}p_{1,i,n}+c_{2,k,n}\big)\\
v_{k,n+1}=v_{k,n}-s_{2,k,n}+q_{2,k,n}
\end{array}
\right.\\[1mm]
\text{For}\;i=1,\ldots,m\\
\left\lfloor
\begin{array}{l}
q_{1,i,n}=p_{1,i,n}-\gamma_n\big(\nabla h_i(p_{1,i,n})+
\sum_{k=1}^KL_{ki}^*p_{2,k,n}+c_{1,i,n}\big)\\
x_{i,n+1}=x_{i,n}-s_{1,i,n}+q_{1,i,n}.
\end{array}
\right.\\
\end{array}
\right.\\
\end{array}
\end{equation}
Then the following hold.
\begin{enumerate}
\item
\label{p:5i}
$(\forall i\!\in\!\{1,\ldots,m\})$
$\sum_{n\in\NN}\|x_{i,n}\!-p_{1,i,n}\|^2\!<\!\pinf$, and
$(\forall k\!\in\{1,\ldots,K\})$
$\sum_{n\in\NN}\|v_{k,n}\!-p_{2,k,n}\|^2\!<\!\pinf$.
\item
\label{p:5ii}
There exist a solution $(\overline{x_1},\ldots,\overline{x_m})$ 
to \eqref{e:5h8Njiq-23p} and a solution 
$(\overline{v_1},\ldots,\overline{v_K})$ to \eqref{e:5h8Njiq-23d} 
such that the following hold.
\begin{enumerate}
\item
\label{p:5iia}
$(\forall i\in\{1,\ldots,m\})$ $x_{i,n}\weakly\overline{x_i}$
~and~ $z_i-\sum_{k=1}^KL_{ki}^*\overline{v_k}\in
\partial f_i(\overline{x_i})+\nabla h_i(\overline{x_i})$.
\item
\label{p:5iib}
$(\forall k\in\{1,\ldots,K\})$ $v_{k,n}\weakly\overline{v_k}$
~and~ $\sum_{i=1}^mL_{ki}\overline{x_i}-r_k\in\partial 
g_k^*(\overline{v_k})+\nabla\ell_k^*(\overline{v_k})$.
\item 
\label{p:5iie}
Suppose that, for some $j\in\{1,\ldots,m\}$, $f_j$ or $h_j$ is 
uniformly convex at $\overline{x_j}$. Then 
$x_{j,n}\to\overline{x_j}$.
\item 
\label{p:5iif}
Suppose that, for some $l\in\{1,\ldots,K\}$, $g^*_l$ or $\ell^*_l$ 
is uniformly convex at $\overline{v_l}$. Then
$v_{l,n}\to\overline{v_l}$.
\end{enumerate}
\end{enumerate}
\end{proposition}
\begin{proof}
Set 
\begin{equation}
\label{e:uH98j3h-17a}
\begin{cases}
(\forall i\in\{1,\ldots,m\})\quad A_i=\partial f_i
\quad\text{and}\quad C_i=\nabla h_i\\
(\forall k\in\{1,\ldots,K\})\quad
B_k=\partial g_k\quad\text{and}\quad D_k=\partial\ell_k. 
\end{cases}
\end{equation}
It follows from \cite[Proposition~17.10]{Livre1} that the 
operators $(C_i)_{1\leq i\leq m}$ are monotone, and from 
\cite[Theorem~20.40]{Livre1} that the operators 
$(A_i)_{1\leq i\leq m}$, $(B_k)_{1\leq k\leq m}$, and 
$(D_k)_{1\leq k\leq K}$ are maximally monotone. Moreover, 
for every $k\in\{1,\ldots,K\}$, we derive from 
\cite[Corollary~13.33 and Theorem~18.15]{Livre1} that $\ell_k^*$ 
is Fr\'echet differentiable on $\GG_k$ and $\nabla\ell^*_k$ is
$\nu_k$-Lipschitzian, and from 
\cite[Corollary~16.24 and Proposition~17.26(i)]{Livre1} that 
$D_k^{-1}=(\partial\ell_k)^{-1}=\partial\ell_k^*=
\{\nabla\ell_k^*\}$. On the other hand, \eqref{e:5h8Njiq-21a}
implies that \eqref{e:IJ834hj8fr-24p} possesses a solution, and 
\eqref{e:prox2} implies that \eqref{e:5h8Njiq-23b} is a 
special case of \eqref{e:5h8Njiq-11b}. 
We also recall that the uniform convexity of a function 
$\varphi\in\Gamma_0(\HH)$ at $x\in\dom\partial\varphi$ implies the 
uniform monotonicity of $\partial\varphi$ at $x$ 
\cite[Section~3.4]{Zali02}. Altogether, the claims will follow at
once from Theorem~\ref{t:1} provided we show that, in the 
setting of \eqref{e:5h8Njiq-21a} and \eqref{e:uH98j3h-17a}, 
\eqref{e:IJ834hj8fr-24p} becomes 
\eqref{e:5h8Njiq-23p} and \eqref{e:IJ834hj8fr-24d} becomes 
\eqref{e:5h8Njiq-23d}. To this end, let us first observe that
since, for every $k\in\{1,\ldots,K\}$, $\dom\,\ell_k^*=\GG_k$, 
\cite[Proposition~24.27]{Livre1} yields 
\begin{equation}
\label{e:IJ834hj8fr-01a}
(\forall k\in\{1,\ldots,K\})\quad
B_k\infconv D_k=
\partial g_k\infconv\partial\ell_k=\partial(g_k\infconv\ell_k),
\end{equation}
while \cite[Corollaries~16.24 and 16.38(iii)]{Livre1} yield 
\begin{equation}
\label{e:IJ834hj8fr-02f}
(\forall k\in\{1,\ldots,K\})\quad B_k^{-1}+D_k^{-1}=
\partial g_k^*+\{\nabla\ell_k^*\}=\partial\big(g_k^*+\ell_k^*\big).
\end{equation}
Likewise, using \cite[Theorem~15.3]{Livre1}, we obtain
\begin{equation}
\label{e:bari3}
(\forall i\in\{1,\ldots,m\})\quad 
(A_i+C_i)^{-1}=(\partial f_i+\nabla h_i)^{-1}
=\big(\partial (f_i+h_i)\big)^{-1}=\partial (f_i+h_i)^*=
\partial (f_i^*\infconv h^*_i).
\end{equation}
Now let us define $\HHH$ and $\GGG$ as in \eqref{e:IJ834hj8fr-30},
$\boldsymbol{L}$, $\boldsymbol{z}$, and $\boldsymbol{r}$ as 
in \eqref{e:5h8Njiq-12b}, and $\boldsymbol{f}$, $\boldsymbol{h}$,
$\boldsymbol{g}$, and $\boldsymbol{\ell}$ as in 
\eqref{e:uH98j3h-20d}.
We derive from \eqref{e:uH98j3h-17a}, \eqref{e:IJ834hj8fr-01a}, 
\cite[Corollary~16.38(iii), Propositions~16.5(ii), 16.8, and 
17.26(i)]{Livre1}, and Fermat's rule \cite[Theorem~16.2]{Livre1}
that, for every $\boldsymbol{x}=(x_i)_{1\leq i\leq m}\in\HHH$,
\begin{align}
\label{e:2012-01-02-md3}
\boldsymbol{x}~\text{solves \eqref{e:IJ834hj8fr-24p}}
&\Leftrightarrow
(\forall i\in\{1,\ldots,m\})\quad 0\in\partial f_i(x_i)
+\sum_{k=1}^KL_{ki}^*\bigg(\partial(g_k\infconv\ell_k)
\bigg(\sum_{j=1}^mL_{kj}x_j-r_k\bigg)\bigg)\nonumber\\
&\quad\hskip 40mm +\nabla h_i(x_i)-z_i\nonumber\\
&\Leftrightarrow
\boldsymbol{0}\in\partial\boldsymbol{f}(\boldsymbol{x})
+\boldsymbol{L}^*\big(\partial(
\boldsymbol{g}\infconv\boldsymbol{\ell})
(\boldsymbol{L}\boldsymbol{x}-\boldsymbol{r})\big)
+\nabla(\boldsymbol{h}-\scal{\cdot}{\boldsymbol{z}})
(\boldsymbol{x})\nonumber\\
&\Rightarrow
\boldsymbol{0}\in\partial\Big(\boldsymbol{f}
+(\boldsymbol{g}\infconv\boldsymbol{\ell})\circ
\big(\boldsymbol{L}\cdot-\boldsymbol{r})\big)
+\boldsymbol{h}-\scal{\cdot}{\boldsymbol{z}}\Big)
(\boldsymbol{x})\nonumber\\
&\Leftrightarrow\boldsymbol{x}~\text{solves \eqref{e:5h8Njiq-23p}.}
\end{align}
Next, let $\boldsymbol{v}=(v_k)_{1\leq k\leq K}\in\GGG$.
Then we derive from \eqref{e:IJ834hj8fr-02f}, \eqref{e:bari3}, 
and the same subdifferential calculus rules as above that 
\begin{align}
\label{e:2012-01-02-md4}
{\boldsymbol v}~\text{solves \eqref{e:IJ834hj8fr-24d}}
&\Leftrightarrow (\forall k\in\{1,\ldots,K\})\quad 0
\in-\sum_{i=1}^mL_{ki}\bigg(\partial(f_i^*\infconv h_i^*)
\bigg(z_i-\sum_{l=1}^KL_{li}^*v_l\bigg)\bigg)\nonumber\\
&\hskip 46mm +\partial\big(g_k^*+\ell_k^*+\scal{\cdot}{r_k}\big)
(v_k)\nonumber\\
&\Leftrightarrow\boldsymbol{0}\in-\boldsymbol{L}
\big(\partial(\boldsymbol{f}^*\infconv \boldsymbol{h}^*)
(\boldsymbol{z}-\boldsymbol{L}^*\boldsymbol{v})\big)
+\partial\big(\boldsymbol{g}^*+\boldsymbol{\ell}^*+
\scal{\cdot}{\boldsymbol{r}}\big)(\boldsymbol{v})\nonumber\\
&\Rightarrow\boldsymbol{0}\in\partial\Big(
(\boldsymbol{f}^*\infconv\boldsymbol{h}^*)\circ
(\boldsymbol{z}-\boldsymbol{L}^*\cdot)
+\boldsymbol{g}^*+\boldsymbol{\ell}^*+
\scal{\cdot}{\boldsymbol{r}}\Big)(\boldsymbol{v})\nonumber\\
&\Leftrightarrow\boldsymbol{\boldsymbol{v}}~
\text{solves \eqref{e:5h8Njiq-23d}},
\end{align}
which completes the proof.
\end{proof}

\begin{remark}
Proposition~\ref{p:5} provides a framework that captures
and suggests extensions of multivariate and/or infimal convolution 
variational formulations found in areas such as partial 
differential equations \cite{Atto11}, machine learning 
\cite{Bach12}, and image recovery \cite{Jmiv11,Caij12,Setz11}.
\end{remark}

\section{Univariate structured convex minimization problems}
\label{sec:7}

Minimization problems involving a single primal variable can be
obtained by setting $m=1$ in Problem~\ref{prob:5}. However, this 
approach imposes that infimal convolutions be performed 
exclusively with strongly convex functions. We use a
different strategy relying on Proposition~\ref{p:1}, which leads
to a formulation allowing for infimal convolutions with general 
lower semicontinuous convex functions.

\begin{problem}
\label{prob:8}
Let $\HH$ be a real Hilbert space, let $K_1$, $K_2$, and $K$ be 
integers such that $0\leq K_1\leq K_2\leq K\geq 1$, let $z\in\HH$, 
let $f\in\Gamma_0(\HH)$, and let $h\colon\HH\to\RR$ be convex and 
differentiable and such that $\nabla h$ is $\mu$-Lipschitzian for 
some $\mu\in\RP$. For every integer $k\in\{1,\ldots,K\}$, let 
$\GG_k$ be a real Hilbert space, let $r_k\in\GG_k$, let 
$g_k\in\Gamma_0(\GG_k)$, let $\varphi_k\in\Gamma_0(\GG_k)$,
and let $L_k\in\BL(\HH,\GG_k)$; moreover, if $K_1+1\leq k\leq K_2$, 
$\varphi_k$ is differentiable on $\GG_k$ and such 
that $\nabla\varphi_k$ is $\beta_k$-Lipschitzian for some 
$\beta_k\in\RP$, and, if $K_2+1\leq k\leq K$, $\varphi_k$ is 
$1/\beta_k$-strongly convex for some $\beta_k\in\RPP$. Set
$\beta=\text{max}\big\{\mu,\beta_{K_1+1},\ldots,\beta_K\big\}+
\sqrt{1+\sum_{k=1}^K\|L_k\|^2}$, and assume that
\begin{equation}
\label{e:kl74Gh9c6-12a}
z\in\ran\Big(\partial f+\sum_{k=1}^KL_{k}^*\circ
(\partial g_k\infconv\partial\varphi_k)
\circ\big(L_{k}\cdot-r_k\big)+\nabla h\Big)
\end{equation}
and
\begin{equation}
\label{e:kl74Gh9c6-12x}
(\forall k\in\{1,\ldots,K_2\})\quad 
0\in\sri(\dom g^*_k-\dom\varphi_k^*).
\end{equation}
Solve the primal problem
\begin{equation}
\label{e:3jqocv9k-12-10p}
\minimize{x\in\HH}{f(x)+\sum_{k=1}^K\,(g_k\infconv\varphi_k)
(L_kx-r_k)+h(x)-\scal{x}{z}}, 
\end{equation}
together with the dual problem
\begin{equation}
\label{e:3jqocv9k-12-10d}
\minimize{v_1\in\GG_1,\ldots,v_K\in\GG_K}{
\big(f^*\infconv h^*\big)\bigg(z-\sum_{k=1}^KL_k^*v_k\bigg)
+\sum_{k=1}^m\big(g_k^*(v_k)+\varphi_k^*(v_k)+
\scal{v_k}{r_k}\big)}.
\end{equation}
\end{problem}

\begin{remark}
\label{r:ultima}
It follows from \eqref{e:kl74Gh9c6-12x} and
\cite[Propositions~11.16, 14.15, 15.7(i), and 24.27]{Livre1} that 
\begin{equation}
\label{e:l}
(\forall k\in\{1,\ldots,K\})\quad g_k\infconv\varphi_k
\in\Gamma_0(\GG_k)\quad\text{and}\quad
\partial g_k\infconv\partial\varphi_k=
\partial(g_k\infconv\varphi_k).
\end{equation}
Hence, using the same type of arguments as in the proof of
Proposition~\ref{p:3jqocv9k-12-10}, we can deduce similar
conditions for \eqref{e:kl74Gh9c6-12a} to hold, e.g., 
that \eqref{e:3jqocv9k-12-10p} have a solution and that
$(r_k)_{1\leq k\leq K}$ lie in the strong relative interior 
of $\menge{(L_{k}x-y_k)_{1\leq k\leq K}}
{x\in\dom f\;\text{and}\;
(\forall k\in\{1,\ldots,K\})\:\:y_k\in
\dom g_k+\dom\varphi_k}$.
\end{remark}

\begin{proposition}
\label{p:8}
Consider the setting of Problem~\ref{prob:8}.
Let $(a_{1,1,n})_{n\in\NN}$, $(b_{1,1,n})_{n\in\NN}$, and 
$(c_{1,1,n})_{n\in\NN}$ be absolutely summable sequences in $\HH$.
For every integer $k\in\{1,\ldots,K\}$, let 
$(a_{2,k,n})_{n\in\NN}$, $(b_{2,k,n})_{n\in\NN}$, and 
$(c_{2,k,n})_{n\in\NN}$ be absolutely summable sequences in 
$\GG_k$; moreover, if $1\leq k\leq K_1$, 
let $(b_{1,k+1,n})_{n\in\NN}$ be 
an absolutely summable sequence in $\GG_k$, and, if
$K_1+1\leq k\leq K_2$\,, let $(a_{1,k+1,n})_{n\in\NN}$ 
and $(c_{1,k+1,n})_{n\in\NN}$ be absolutely summable sequences in 
$\GG_k$. Let $x_{0}\in\HH$, $y_{1,0}\in\GG_1$, 
\ldots, $y_{K_2,0}\in\GG_{K_2}$, $v_{1,0}\in\GG_1$, \ldots, and
$v_{K,0}\in\GG_K$, let $\varepsilon\in\left]0,1/(\beta+1)\right[$,
let $(\gamma_n)_{n\in\NN}$ be a sequence in 
$[\varepsilon,(1-\varepsilon)/\beta]$, and set
\begin{equation}
\label{e:3jqocv9k-12-10g}
\begin{array}{l}
\text{For}\;n=0,1,\ldots\\
\left\lfloor
\begin{array}{l}
s_{1,1,n}=x_{n}-\gamma_n\big(\nabla h(x_n)+
\sum_{k=1}^KL_k^*v_{k,n}+a_{1,1,n}\big)\\
p_{1,1,n}=\prox_{\gamma_n f}(s_{1,1,n}+\gamma_nz)+b_{1,1,n}\\
\text{If}\;K_1\neq 0,\;\text{for}\;k=1,\ldots,K_1\\
\left\lfloor
\begin{array}{l}
s_{1,k+1,n}=y_{k,n}+\gamma_nv_{k,n}\\
p_{1,k+1,n}=\prox_{\gamma_n\varphi_k}s_{1,k+1,n}+b_{1,k+1,n}\\
s_{2,k,n}=v_{k,n}-\gamma_n\big(y_{k,n}-L_kx_n+a_{2,k,n}\big)\\
p_{2,k,n}=s_{2,k,n}-\gamma_n\big(r_k+\prox_{\gamma_n^{-1}g_k}
(\gamma_n^{-1}s_{2,k,n}-r_k)+b_{2,k,n}\big)\\
q_{2,k,n}=p_{2,k,n}-\gamma_n\big(p_{1,k+1,n}-L_kp_{1,1,n}
+c_{2,k,n}\big)\\
v_{k,n+1}=v_{k,n}-s_{2,k,n}+q_{2,k,n}
\end{array}
\right.\\[1mm]
\text{If}\;K_1\neq K_2,\;\text{for}\;k=K_1+1,\ldots,K_2\\
\left\lfloor
\begin{array}{l}
s_{1,k+1,n}=y_{k,n}-\gamma_n\big(\nabla\varphi_k(y_{k,n})-v_{k,n}
+a_{1,k+1,n}\big)\\
p_{1,k+1,n}=s_{1,k+1,n}\\
s_{2,k,n}=v_{k,n}-\gamma_n\big(y_{k,n}-L_kx_n
+a_{2,k,n}\big)\\
p_{2,k,n}=s_{2,k,n}-\gamma_n\big(r_k+\prox_{\gamma_n^{-1}g_k}
(\gamma_n^{-1}s_{2,k,n}-r_k)+b_{2,k,n}\big)\\
q_{2,k,n}=p_{2,k,n}-\gamma_n\big(p_{1,k+1,n}-L_kp_{1,1,n}
+c_{2,k,n}\big)\\
v_{k,n+1}=v_{k,n}-s_{2,k,n}+q_{2,k,n}
\end{array}
\right.\\[1mm]
\text{If}\;K_2\neq K,\;\text{for}\;k=K_2+1,\ldots,K\\
\left\lfloor
\begin{array}{l}
s_{2,k,n}=v_{k,n}-\gamma_n\big(\nabla\varphi_k^*(v_{k,n})-L_kx_n
+a_{2,k,n}\big)\\
p_{2,k,n}=s_{2,k,n}-\gamma_n\big(r_k+\prox_{\gamma_n^{-1}g_k}
(\gamma_n^{-1}s_{2,k,n}-r_k)+b_{2,k,n}\big)\\
q_{2,k,n}=p_{2,k,n}-\gamma_n\big(\nabla\varphi_k^*
(p_{2,k,n})-L_kp_{1,1,n}+c_{2,k,n}\big)\\
v_{k,n+1}=v_{k,n}-s_{2,k,n}+q_{2,k,n}
\end{array}
\right.\\[1mm]
q_{1,1,n}=p_{1,1,n}-\gamma_n\big(\nabla h(p_{1,1,n})+
\sum_{k=1}^KL_k^*p_{2,k,n}+c_{1,1,n}\big)\\
x_{n+1}=x_{n}-s_{1,1,n}+q_{1,1,n}\\
\text{If}\;K_1\neq 0,\;\text{for}\;k=1,\ldots,K_1\\
\left\lfloor
\begin{array}{l}
q_{1,k+1,n}=p_{1,k+1,n}+\gamma_np_{2,k,n}\\
y_{k,n+1}=y_{k,n}-s_{1,k+1,n}+q_{1,k+1,n}
\end{array}
\right.\\
\text{If}\;K_1\neq K_2,\;\text{for}\;k=K_1+1,\ldots,K_2\\
\left\lfloor
\begin{array}{l}
q_{1,k+1,n}=p_{1,k+1,n}-\gamma_n\big(\nabla\varphi_k(p_{1,k+1,n})
-p_{2,k,n}+c_{1,k+1,n}\big)\\
y_{k,n+1}=y_{k,n}-s_{1,k+1,n}+q_{1,k+1,n}.
\end{array}
\right.\\
\end{array}
\right.\\
\end{array}
\end{equation}
Then the following hold for some solution $\overline{x}$ to 
\eqref{e:3jqocv9k-12-10p} and some solution 
$(\overline{v_1},\ldots,\overline{v_K})$ to 
\eqref{e:3jqocv9k-12-10d}.
\begin{enumerate}
\item
\label{p:8i}
$x_n\weakly\overline{x}~$ and 
$~(\forall k\in\{1,\ldots,K\})$ $v_{k,n}\weakly\overline{v_k}$.
\item 
\label{p:8iii}
Suppose that $f$ or $h$ is uniformly convex at 
$\overline{x}$. Then $x_{n}\to\overline{x}$.
\item 
\label{p:8iv}
Suppose that, for some $l\in\{1,\ldots,K\}$, $g_l^*$ 
is uniformly convex at $\overline{v_l}$. Then
$v_{l,n}\to\overline{v_l}$.
\item 
\label{p:8v}
Suppose that $K_2\neq K$ and that, for some 
$l\in\{K_2+1,\ldots,K\}$, $\varphi_l^*$ is uniformly convex at 
$\overline{v_l}$. Then $v_{l,n}\to\overline{v_l}$.
\end{enumerate}
\end{proposition}
\begin{proof}
Using \eqref{e:l} and the same arguments as in the proof of 
Proposition~\ref{p:5}, we first identify Problem~\ref{prob:8} 
as a special case of Problem~\ref{prob:2} with $A=\partial f$, 
$C=\nabla h$, and $(\forall k\in\{1,\ldots,K\})$
$B_k=\partial g_k$ and $S_k=\partial\varphi_k$. Using 
\eqref{e:prox2}, we then deduce the results from 
Proposition~\ref{p:1}.
\end{proof}

We conclude this section with an application to the approximation
of inconsistent convex feasibility problems where, for the sake
of brevity, we discuss only the primal problem.

\begin{example}
\label{p:kl74Gh9c6-23}
In Problem~\ref{prob:8}, set $K_1=K_2=K$, $z=0$, $h=0$, $f=0$, and,
for every $k\in\{1,\ldots,K\}$ $r_k=0$ and $g_k=\iota_{C_k}$, 
where $C_k$ is a nonempty closed convex subset of $\GG_k$ with 
projection operator $P_k$. In addition, suppose that 
\begin{equation}
\label{e:kl74Gh9c6-12y}
(\forall k\in\{1,\ldots,K\})\quad\text{Argmin}\,\varphi_k=\{0\},
\;\varphi_k(0)=0,\;\;\text{and}\;\;
0\in\sri(\dom \iota^*_{C_k}-\dom\varphi_k^*).
\end{equation}
It follows from \cite[Proposition~15.7(i)]{Livre1} that the infimal 
convolutions $(\iota_{C_k}\infconv\varphi_k)_{1\leq k\leq K}$ are 
exact. Hence, \eqref{e:3jqocv9k-12-10p} becomes
\begin{equation}
\label{e:kl74Gh9c6-23p}
\minimize{x\in\HH}{\sum_{k=1}^{K}\underset{y_k\in C_k}{\text{min}}
\varphi_k(L_kx-y_k)},
\end{equation}
and it is assumed to have at least one solution. We can interpret
\eqref{e:kl74Gh9c6-23p} as a relaxation of the (possibly
inconsistent) convex feasibility problem
\begin{equation}
\label{e:kl74Gh9c6-23f}
\text{find}\;\;\overline{x}\in\HH\;\;\text{such that}\;\;
(\forall k\in\{1,\ldots,K\})\quad L_k\overline{x}\in C_k.
\end{equation}
Indeed, it follows from \eqref{e:kl74Gh9c6-12y} that, if 
\eqref{e:kl74Gh9c6-23f} is consistent, then its solutions coincide 
with those of \eqref{e:kl74Gh9c6-23p}.
Furthermore, in view of \eqref{e:prox2}, Algorithm 
\eqref{e:3jqocv9k-12-10g} can be written as
\begin{equation}
\label{e:3jqocv9k-12-10j}
\begin{array}{l}
\text{For}\;n=0,1,\ldots\\
\left\lfloor
\begin{array}{l}
p_{1,1,n}=x_{n}-\gamma_n\big(
\sum_{k=1}^KL_k^*v_{k,n}+a_{1,1,n}\big)\\
\text{For}\;k=1,\ldots,K\\
\left\lfloor
\begin{array}{l}
s_{1,k+1,n}=y_{k,n}+\gamma_nv_{k,n}\\
p_{1,k+1,n}=\prox_{\gamma_n \varphi_k}s_{1,k+1,n}+b_{1,k+1,n}\\
s_{2,k,n}=v_{k,n}-\gamma_n\big(y_{k,n}-L_kx_n+a_{2,k,n}\big)\\
p_{2,k,n}=s_{2,k,n}-\gamma_n\big(P_{k}
(\gamma_n^{-1}s_{2,k,n})+b_{2,k,n}\big)\\
q_{2,k,n}=p_{2,k,n}-\gamma_n\big(p_{1,k+1,n}-L_kp_{1,1,n}
+c_{2,k,n}\big)\\
v_{k,n+1}=v_{k,n}-s_{2,k,n}+q_{2,k,n}
\end{array}
\right.\\[1mm]
q_{1,1,n}=p_{1,1,n}-\gamma_n\big(
\sum_{k=1}^KL_k^*p_{2,k,n}+c_{1,1,n}\big)\\
x_{n+1}=x_{n}-p_{1,1,n}+q_{1,1,n}\\
\text{For}\;k=1,\ldots,K\\
\left\lfloor
\begin{array}{l}
q_{1,k+1,n}=p_{1,k+1,n}+\gamma_np_{2,k,n}\\
y_{k,n+1}=y_{k,n}-s_{1,k+1,n}+q_{1,k+1,n}.
\end{array}
\right.\\
\end{array}
\right.\\
\end{array}
\end{equation}
By Proposition~\ref{p:8}\ref{p:8i}, $(x_n)_{n\in\NN}$ converges 
weakly to a solution to \eqref{e:kl74Gh9c6-23p} if
$\text{inf}_{n\in\NN}\,\gamma_n>0$ and 
$\text{sup}_{n\in\NN}\,\gamma_n<
\big({1+\sum_{k=1}^K\|L_k\|^2}\big)^{-1/2}$.
Now suppose that, for every $k\in\{1,\ldots,K\}$, 
$\GG_k=\HH$, $L_k=\Id$, $\varphi_k=\iota_{\{0\}}$ if $k=1$, and
$\varphi_k=\omega_k\|\cdot\|^2$, where $\omega_k\in\RPP$, if
$k\neq 1$. Then \eqref{e:kl74Gh9c6-23f} reduces to the 
feasibility problem of finding $\overline{x}\in\bigcap_{k=1}^KC_k$
and \eqref{e:kl74Gh9c6-23p} reduces to the 
constrained least-squares relaxation studied in 
\cite{Sign99}, namely,
$\minimize{x\in C_1}{\sum_{k=2}^{K}\omega_kd_{C_k}^2(x)}$.
\end{example}

\end{document}